\documentclass[10pt]{article}  \usepackage{amsmath, amsthm, amssymb, enumerate} \usepackage[unicode,breaklinks=true,colorlinks=true,linkcolor=blue,urlcolor=blue,citecolor=blue]{hyperref}  \usepackage{marginnote} \usepackage{tikz}  \usepackage{color} \usepackage{xcolor} \usepackage{datetime} \usepackage{fancyhdr,lastpage} \chardef\coloryes=0     \chardef\isitdraft=0     \ifnum\isitdraft=1    https://preview.overleaf.com/public/tfqhbhwccrwh/images/4f4289976560863c3073d36f0bc870df24fe1ade.jpeg \textwidth 16truecm \textheight 8.4in\oddsidemargin0.2truecm\evensidemargin0.7truecm\voffset-.9truecm         \def\eqref#1{({\ref{#1}})}                 \def\startnewsection#1#2{\section{#1}\label{#2}\setcounter{equation}{0}}    \textwidth 16truecm \textheight 8.4in\oddsidemargin0.2truecm\evensidemargin0.7truecm\voffset-.9truecm \def\nnewpage{}  \else    \def\startnewsection#1#2{\section{#1}\label{#2}\setcounter{equation}{0}}    \textwidth 16truecm \textheight 8.4in\oddsidemargin0.2truecm\evensidemargin0.7truecm\voffset-.9truecm \def\nnewpage{}  \fi    \begin{document} \nocite{*} \def\ques{{\colr \underline{??????}\colb}} \def\nto#1{{\colC \footnote{\em \colC #1}}} \def\fractext#1#2{{#1}/{#2}} \def\fracsm#1#2{{\textstyle{\frac{#1}{#2}}}}    \def\nnonumber{} \def\les{\lesssim} \def\plusdelta{+\delta} \def\colr{{}}    \def\colg{{}}    \def\colb{{}}    \def\colu{{}}    \def\cole{{}}    \def\colA{{}}    \def\colB{{}}    \def\colC{{}}    \def\colD{{}}    \def\colE{{}}    \def\colF{{}}    \ifnum\coloryes=1    \definecolor{coloraaaa}{rgb}{0.1,0.2,0.8}    \definecolor{colorbbbb}{rgb}{0.1,0.7,0.1}    \definecolor{colorcccc}{rgb}{0.8,0.3,0.9}    \definecolor{colordddd}{rgb}{0.0,.5,0.0}    \definecolor{coloreeee}{rgb}{0.8,0.3,0.9}    \definecolor{colorffff}{rgb}{0.8,0.3,0.9}    \definecolor{colorgggg}{rgb}{0.5,0.0,0.4}    \def\colb{\color{black}}    \def\colr{\color{red}}    \def\cole{\color{colorgggg}}    \def\colu{\color{blue}}    \def\colg{\color{colordddd}}    \def\colA{\color{coloraaaa}}    \def\colB{\color{colorbbbb}}    \def\colC{\color{colorcccc}}    \def\colD{\color{colordddd}}    \def\colE{\color{coloreeee}}    \def\colF{\color{colorffff}}    \def\colG{\color{colorgggg}}    \fi    \ifnum\isitdraft=1    \chardef\coloryes=1     \baselineskip=17pt    \pagestyle{myheadings} \reversemarginpar \def\const{\mathop{\rm const}\nolimits}   \def\diam{\mathop{\rm diam}\nolimits}     \def\rref#1{{\ref{#1}{\rm \tiny \fbox{\tiny #1}}}} \def\theequation{\fbox{\bf \thesection.\arabic{equation}}} \def\plusdelta{+\delta} \def\startnewsection#1#2{\newpage\colg \section{#1}\colb\label{#2} \setcounter{equation}{0} \pagestyle{fancy} \lhead{\colb Section~\ref{#2}, #1 } \cfoot{} \rfoot{\thepage\ of \pageref{LastPage}} \lfoot{\colb{\today,~\currenttime}~(kw1)}} \chead{} \rhead{\thepage} \def\nnewpage{\newpage} \newcounter{startcurrpage} \newcounter{currpage} \def\llll#1{{\rm\tiny\fbox{#1}}} \def\blackdot{{\color{red}{\hskip-.0truecm\rule[-1mm]{4mm}{4mm}\hskip.2truecm}}\hskip-.3truecm}    \def\bluedot{{\colC {\hskip-.0truecm\rule[-1mm]{4mm}{4mm}\hskip.2truecm}}\hskip-.3truecm}    \def\purpledot{{\colA{\rule[0mm]{4mm}{4mm}}\colb}}    \def\pdot{\purpledot}    \else      \baselineskip=15pt \def\blackdot{{\rule[-3mm]{8mm}{8mm}}}    \def\blackdot{{\color{red}{\hskip-.0truecm\rule[-1mm]{4mm}{4mm}\hskip.2truecm}}\hskip-.3truecm}    \def\purpledot{{\rule[-3mm]{8mm}{8mm}}}    \def\pdot{} \fi    \def\SS{S} \def\bSS{\bar S} \def\qq{{\bar q}} \def\KK{K} \def\ema#1{\underline{\underline{#1}}} \def\emb#1{\dotuline{#1}} \ifnum\isitdraft=1 \def\llabel#1{\marginnote{\small\colb\hskip1.5truecm\boxed{#1}}[-0.7cm]\nonumber} \else \def\llabel#1{\nonumber} \fi \def\tepsilon{\tilde\epsilon} \def\restr{\bigm|} \def\into{\int_{\Omega}} \def\intu{\int_{\Gamma_1}} \def\intl{\int_{\Gamma_0}} \def\tpar{\tilde\partial} \def\bpar{\,|\nabla_2|} \def\barpar{\bar\partial} \def\FF{F} \def\gdot{{\color{green}{\hskip-.0truecm\rule[-1mm]{4mm}{4mm}\hskip.2truecm}}\hskip-.3truecm}    \def\bdot{{\color{blue}{\hskip-.0truecm\rule[-1mm]{4mm}{4mm}\hskip.2truecm}}\hskip-.3truecm}    \def\cydot{{\color{cyan} {\hskip-.0truecm\rule[-1mm]{4mm}{4mm}\hskip.2truecm}}\hskip-.3truecm}    \def\rdot{{\color{red} {\hskip-.0truecm\rule[-1mm]{4mm}{4mm}\hskip.2truecm}}\hskip-.3truecm}    \def\tdot{\fbox{\fbox{\bf\color{blue}\tiny I'm here; \today \ \currenttime}}} \def\nts#1{{\color{red}\hbox{\bf ~#1~}}}  \def\ntsr#1{\vskip.0truecm{\color{red}\hbox{\bf ~#1~}}\vskip0truecm}  \def\ntsf#1{\footnote{\hbox{\bf ~#1~}}}  \def\ntsf#1{\footnote{\color{blue}\hbox{\bf ~#1~}}}  \def\bigline#1{~\\\hskip2truecm~~~~{#1}{#1}{#1}{#1}{#1}{#1}{#1}{#1}{#1}{#1}{#1}{#1}{#1}{#1}{#1}{#1}{#1}{#1}{#1}{#1}{#1}\\}    \def\biglineb{\bigline{$\downarrow\,$ $\downarrow\,$}}    \def\biglinem{\bigline{---}}    \def\biglinee{\bigline{$\uparrow\,$ $\uparrow\,$}}    \def\ceil#1{\lceil #1 \rceil} \def\gdot{{\color{green}{\hskip-.0truecm\rule[-1mm]{4mm}{4mm}\hskip.2truecm}}\hskip-.3truecm}    \def\bluedot{{\color{blue} {\hskip-.0truecm\rule[-1mm]{4mm}{4mm}\hskip.2truecm}}\hskip-.3truecm}    \def\rdot{{\color{red} {\hskip-.0truecm\rule[-1mm]{4mm}{4mm}\hskip.2truecm}}\hskip-.3truecm}    \def\dbar{\bar{\partial}} \newtheorem{Theorem}{Theorem}[section] \newtheorem{Corollary}[Theorem]{Corollary} \newtheorem{Proposition}[Theorem]{Proposition} \newtheorem{Lemma}[Theorem]{Lemma} \theoremstyle{remark} \newtheorem{Remark}[Theorem]{Remark} \newtheorem{definition}{Definition}[section] \def\theequation{\thesection.\arabic{equation}} \def\sqrtg{\sqrt{g}} \def\dd{\delta} \def\EE{{\mathcal E}} \def\lot{{\rm l.o.t.}}                        \def\inon#1{~~~\hbox{#1}}                 \def\endproof{\hfill$\Box$\\} \def\square{\hfill$\Box$\\} \def\inon#1{~~~\hbox{#1}}                 \def\comma{ {\rm ,\qquad{}} }             \def\commaone{ {\rm ,\qquad{}} }          \def\dist{\mathop{\rm dist}\nolimits}     \def\sgn{\mathop{\rm sgn\,}\nolimits}     \def\Tr{\mathop{\rm Tr}\nolimits}     \def\dive{\mathop{\rm div}\nolimits}     \def\grad{\mathop{\rm grad}\nolimits}    \def\curl{\mathop{\rm curl}\nolimits}    \def\det{\mathop{\rm det}\nolimits}     \def\supp{\mathop{\rm supp}\nolimits}   \def\re{\mathop{\rm {\mathbb R}e}\nolimits} \def\wb{\bar{\omega}} \def\Wb{\bar{W}} \def\indeq{\qquad{}}                      \def\indeqtimes{\indeq\indeq\times}  \def\period{.}                            \def\semicolon{\,;}                       \newcommand{\cD}{\mathcal{D}} \newcommand{\cH}{\mathcal{H}} \newcommand{\al}{\alpha} \newcommand{\be}{\beta} \newcommand{\ga}{\gamma} \newcommand{\de}{\delta} \newcommand{\ep}{\epsilon} \newcommand{\si}{\sigma} \newcommand{\Si}{\Sigma} \newcommand{\vfi}{\varphi} \newcommand{\om}{\omega} \newcommand{\Om}{\Omega} \newcommand{\cqd}{\hfill $\qed$\\ \medskip} \newcommand{\rar}{\rightarrow} \newcommand{\imp}{\Rightarrow} \newcommand{\tr}{\operatorname{tr}}  
\def\bal#1{\heyu{#1}} \def\lla#1{} \newcommand{\vol}{\operatorname{vol}} \newcommand{\id}{\operatorname{id}} \newcommand{\p}{\parallel} \newcommand{\norm}[1]{\Vert#1\Vert} \newcommand{\abs}[1]{\vert#1\vert} \def\heyu#1{\lla{#1}} \newcommand{\nnorm}[1]{\left\Vert#1\right\Vert} \newcommand{\aabs}[1]{\left\vert#1\right\vert} \title{Global Sobolev persistence for the fractional Boussinesq equations with zero diffusivity} \author{Igor Kukavica and Weinan Wang}  \date{} \maketitle \bigskip \indent Department of Mathematics\\ \indent University of Southern California\\ \indent Los Angeles, CA 90089\\ \indent e-mails: kukavica@usc.edu, wangwein@usc.edu \bigskip \begin{abstract} We address the persistence of regularity for the 2D $\alpha$-fractional Boussinesq equations with positive viscosity and zero diffusivity in general Sobolev spaces, i.e., for $(u_{0}, \rho_{0}) \in W^{s,q}(\mathbb R^2) \times W^{s,q}(\mathbb R^2)$, where $s> 1$ and $q \in (2, \infty)$. We prove that the solution $(u(t), \rho(t))$ exists and belongs to $W^{s,q}(\mathbb R^2) \times W^{s,q}(\mathbb R^2)$  for all positive time $t$ for $q>2$, where $\alpha\in(1,2)$ is arbitrary. \end{abstract} \startnewsection{Introduction}{sec01}  In this paper, we address the persistence of regularity for the 2D fractional Boussinesq equations   \begin{align}    & u_t     + \Lambda^{\alpha}u     + u\cdot \nabla u     + \nabla \pi     = \rho e_2     \llabel{EQ55}    \\     & \rho_t + u\cdot \nabla \rho = 0    \llabel{EQ61}    \\    & \nabla \cdot u=0    \llabel{EQ67}   \end{align}  in Sobolev spaces.  Here, $u$ is the velocity satisfying the 2D Navier-Stokes equations \cite{CF, DG, FMT, R, T2, T3}  driven by $\rho$, which represents the density or temperature of the fluid, depending on the physical context. Also, $e_2 = (0, 1)$ is the unit vector in the vertical direction and $1<\alpha<2$.  \par\heyu{ 3gBo 5aBKK5 gf J SmN eCW wOM t9 xutz wDkX IY7nNh Wd D ppZ UOq 2Ae 0a W7A6 XoIc TSLNDZ yf 2 XjB cUw eQT Zt cuXI DYsD hdAu3V MB B BKW IcF NWQ dO u3Fb c6F8 VN77Da IH E 3MZ luL YvB mN Z2wE auXX DGpeKR nw o UVB 2oM} The global existence and persistence of regularity has been a topic of high interest since the seminal works of  Chae \cite{C} and of Hou and Li \cite{HL}, who proved the global existence of a unique solution in the case of Laplacian, $\alpha=2$. Namely, the global persistence holds for $(u_0,\rho_0)$ in $H^{s}\times H^{s-1}$ for integers $s\ge3$ \cite{HL},  while  we have the global persistence in  $H^{s}\times H^{s}$ for integers $s\ge3$ by \cite{C}. The global existence and uniqueness in the low regularity space $H^{1}\times L^2$ was established by Lunasin~et~al in \cite{LLT}. The persistence in $H^{s}\times H^{s-1}$ for the intermediate values $1< s< 3$ was then settled in \cite{HKZ1, HKZ2}. For other results on the global existence and persistence of solutions, cf.~\cite{ACW, BS,  BrS, CD, CG, CLR, CN, CW, DP1, DP2, ES, HK1, HK2, HS, JMWZ, KTW, KWZ, LPZ, SW, T1}. \par\heyu{ VVe hW 0ejG gbgz Iw9FwQ hN Y rFI 4pT lqr Wn Xzz2 qBba lv3snl 2j a vzU Snc pwh cG J0Di 3Lr3 rs6F23 6o b LtD vN9 KqA pO uold 3sec xqgSQN ZN f w5t BGX Pdv W0 k6G4 Byh9 V3IicO nR 2 obf x3j rwt 37 u82f wxwj SmOQq0 pq} The main difficulty when studying the persistence of regularity in the Sobolev spaces $W^{s,q}\times W^{s-1,q}$ when $q>2$ is the lack of availability of the energy equation, which is one of the essential features of the Boussinesq system. This problem was studied in~\cite{KWZ}, where it was proven that the persistence holds if $(s-1)q > 2$.  \par\heyu{ 4 qfv rN4 kFW hP HRmy lxBx 1zCUhs DN Y INv Ldt VDG 35 kTMT 0ChP EdjSG4 rW N 6v5 IIM TVB 5y cWuY OoU6 Sevyec OT f ZJv BjS ZZk M6 8vq4 NOpj X0oQ7r vM v myK ftb ioR l5 c4ID 72iF H0VbQz hj H U5Z 9EV MX8 1P GJss} In the present paper, we consider the fractional dissipation in the range $1<\alpha < 2$, addressing the persistence  in $W^{s,q}({\mathbb R}^2)\times W^{s,q}({\mathbb R}^2)$. Namely, we prove that if $(u_0,\rho_0)\in W^{s,q}({\mathbb R}^2)\times W^{s,q}({\mathbb R}^2)$, then  $(u(\cdot,t),\rho(\cdot,t))\in W^{s,q}({\mathbb R}^2)\times W^{s,q}({\mathbb R}^2)$ for all $t\ge 0$. The main result is contained in Theorem~\ref{T01} and asserts the global persistence for all $s>1$. The main device in the proof is the generalized vorticity   \begin{equation}    \zeta=\omega-\partial_{1}(I-\Delta)^{-\alpha/2}\rho       .    \label{EQ103}   \end{equation} This change of variable is inspired by the one introduced by  Jiu et al in \cite{JMWZ}, (cf.~also \cite{SW}), which in turn drew from the work of Hmidi, Keraani, and Rousset \cite{HKR}. Here we need to modify it to avoid problems with low frequencies as our data are not square integrable. We show in \eqref{EQ08} below that the modified vorticity $\zeta$  defined in \eqref{EQ103} satisfies the equation   \begin{align}   \begin{split}   \zeta_{t}   +   u\cdot \nabla \zeta   +   \Lambda^{\alpha} \zeta   =   [\SS, u\cdot \nabla]\rho    -   (\tilde \Lambda^{-\alpha} \Lambda^{\alpha} - I)\partial_1 \rho   \end{split}    \label{EQ133}   \end{align} where $   \SS   = \partial_{1}(I-\Delta)^{-\alpha/2} $ with $ \Lambda=(-\Delta)^{1/2}$ and $\tilde \Lambda=(I-\Delta)^{1/2}$. Compared to the original change of variable in  \cite{JMWZ}, we obtain a new term  $N \rho=  (\tilde \Lambda^{-\alpha} \Lambda^{\alpha} - I)\partial_1 \rho$, for which however we show  in Lemma~\ref{L01} below that it is smoothing of degree 1. The reason why this change of variable is suitable for low frequencies is due to the inhomogeneity in the  second term of \eqref{EQ103}. \par\heyu{ Wedm hBXKDA iq w UJV Gj2 rIS 92 AntB n1QP R3tTJr Z1 e lVo iKU stz A8 fCCg Mwfw 4jKbDb er B Rt6 T8O Zyn NO qXc5 3Pgf LK9oKe 1p P rYB BZY uui Cw XzA6 kaGb twGpmR Tm K viw HEz Rjh Te frip vLAX k3PkLN Dg 5 odc omQ} Also, an important part of the proof of Sobolev persistence is based on the observation that a fractional derivative of the commutator term in \eqref{EQ133} is a sum of two terms, which are also  of commutator type and are thus suitable for the use of a Kato-Ponce type inequality; cf.~\eqref{EQ53}  and Remark~\ref{R01} below. \par\heyu{ j9L YI VawV mLpK rto0F6 Ns 7 Mmk cTL 9Tr 8f OT4u NNJv ZThOQw CO C RBH RTx hSB Na Iizz bKIB EcWSMY Eh D kRt PWG KtU mo 26ac LbBn I4t2P1 1e R iPP 99n j4q Q3 62UN AQaH JPPY1O gL h N8s ta9 eJz Pg mE4z QgB0 mlAWBa 4E} \par\heyu{ m u7m nfY gbN Lz ddGp hhJV 9hyAOG CN j xJ8 3Hg 6CA UT nusW 9pQr Wv1DfV lG n WxM Bbe 9Ww Lt OdwD ERml xJ8LTq KW T tsR 0cD XAf hR X1zX lAUu wzqnO2 o7 r toi SMr OKL Cq joq1 tUGG iIxusp oi i tja NRn gtx S0 r98r} \par\heyu{ wXF7 GNiepz Ef A O2s Ykt Idg H1 AGcR rd2w 89xoOK yN n LaL RU0 3su U3 JbS8 dok8 tw9NQS Y4 j XY6 25K CcP Ly FRlS p759 DeVbY5 b6 9 jYO mdf b99 j1 5lvL vjsk K2gEwl Rx O tWL ytZ J1y Z5 Pit3 5SOi ivz4F8 tq M JIg QQi} \par\heyu{ Oob Sp eprt 2vBV qhvzkL lf 7 HXA 4so MXj Wd MS7L eRDi ktUifL JH u kes trv rl7 mY cSOB 7nKW MD0xBq kb x FgT TNI wey VI G6Uy 3dL0 C3MzFx sB E 7zU hSe tBQ cX 7jn2 2rr0 yL1Erb pL R m3i da5 MdP ic dnMO iZCy Gd2MdK Ub} \par\heyu{ x saI 9Tt nHX qA QBju N5I4 Q6zz4d SW Y Urh xTC uBg BU T992 uczE mkqK1o uC a HJB R0Q nv1 ar tFie kBu4 9ND9kK 9e K BOg PGz qfK J6 7NsK z3By wIwYxE oW Y f6A Kuy VPj 8B 9D6q uBkF CsKHUD Ck s DYK 3vs 0Ep 3g M2Ew} The paper is organized as follows.  In Section~\ref{sec02}, we state the main theorem on the persistence and introduce the change of the vorticity variable. We also prove the smoothing property of the operator $N$. The next section contains a variant of a Kato-Ponce lemma suitable for the operator  $\SS$ arising in \eqref{EQ133}. Lemma~\ref{L06} contains the bound for the vorticity and its modified version~$\zeta$. The proof of the main theorem for the case $s\leq\alpha $ is then provided in Section~\ref{sec04}. Finally, the last section contains the proof of the main theorem for $s>\alpha$. This part of the proof requires the case $s\leq \alpha$ when we establish a bound on $\Vert \Lambda^{1/2}u \Vert_{L^\infty}$ in \eqref{EQ108} below. \par\heyu{ lPGj RVX6cx lb V OfA ll7 g6y L9 PWyo 58h0 e07HO0 qz 8 kbe 85Z BVC YO KxNN La4a FZ7mw7 mo A CU1 q1l pfm E5 qXTA 0QqV MnRsbK zH o 5vX 1tp MVZ XC znmS OM73 CRHwQP Tl v VN7 lKX I06 KT 6MTj O3Yb 87pgoz ox y dVJ HPL} \par\heyu{ 3k2 KR yx3b 0yPB sJmNjE TP J i4k m2f xMh 35 MtRo irNE 9bU7lM o4 b nj9 GgY A6v sE sONR tNmD FJej96 ST n 3lJ U2u 16o TE Xogv Mqwh D0BKr1 Ci s VYb A2w kfX 0n 4hD5 Lbr8 l7Erfu N8 O cUj qeq zCC yx 6hPA yMrL eB8Cwl kT} \par\heyu{ h ixd Izv iEW uw I8qK a0VZ EqOroD UP G phf IOF SKZ 3i cda7 Vh3y wUSzkk W8 S fU1 yHN 0A1 4z nyPU Ll6h pzlkq7 SK N aFq g9Y hj2 hJ 3pWS mi9X gjapmM Z6 H V8y jig pSN lI 9T8e Lhc1 eRRgZ8 85 e NJ8 w3s ecl 5i lCdo} \par\heyu{ zV1B oOIk9g DZ N Y5q gVQ cFe TD VxhP mwPh EU41Lq 35 g CzP tc2 oPu gV KOp5 Gsf7 DFBlek to b d2y uDt ElX xm j1us DJJ6 hj0HBV Fa n Tva bFA VwM 51 nUH6 0GvT 9fAjTO 4M Q VzN NAQ iwS lS xf2p Q8qv tdjnvu pL A TIw ym4} \startnewsection{Notation and the main result on global persistence }{sec02} We consider solutions of the Boussinesq system   \begin{align}    & u_t     + \Lambda^{\alpha}u     + u\cdot \nabla u     + \nabla \pi     = \rho e_2    \label{EQ01}   \\    & \rho_t + u\cdot \nabla \rho = 0   \label{EQ02}   \\   & \nabla \cdot u=0   ,   \label{EQ03}   \end{align}  where the operator $\Lambda^{\alpha}$ is defined by   \begin{equation*}    \Lambda^{\alpha} = (-\Delta)^{\alpha/2}    \comma 1<\alpha<2    ,   \end{equation*} or, using the Fourier transform,   \begin{equation}    (\Lambda^\alpha f) \hat~(\xi) = |\xi|^{\alpha} \hat f(\xi)    \comma \xi\in{\mathbb R}^2       .   \label{EQ04}   \end{equation} The following is the main result of the paper. \par\heyu{ nEY ES fMav UgZo yehtoe 9R T N15 EI1 aKJ SC nr4M jiYh B0A7vn SA Y nZ1 cXO I1V 7y ja0R 9jCT wxMUiM I5 l 2sT XnN RnV i1 KczL G3Mg JoEktl Ko U 13t saq jrH YV zfb1 yyxu npbRA5 6b r W45 Iqh fKo 0z j04I cGrH irwyH2 tJ} \cole \begin{Theorem} \label{T01} Let $q \in (2,\infty)$ and $s>1$. Assume that  $\Vert u_0 \Vert_{W^{s,q}} < \infty$ with  $\nabla \cdot u_0 = 0$ and $\Vert \rho_0 \Vert_{W^{s,q}} < \infty$.  Then there exists a unique solution $(u, \rho)$ to the equations \eqref{EQ01}--\eqref{EQ03} such that  $u \in C\left([0,T], W^{s,q}(\mathbb R^2)\right)$ and $\rho \in C\left([0,T], W^{s,q}(\mathbb R^2)\right)$ for all $T>0$. \end{Theorem} \colb \par\heyu{ b Fr3 leR dcp st vXe2 yJle kGVFCe 2a D 4XP OuI mtV oa zCKO 3uRI m2KFjt m5 R GWC vko zi7 5Y WNsb hORn xzRzw9 9T r Fhj hKb fqL Ab e2v5 n9mD 2VpNzl Mn n toi FZB 2Zj XB hhsK 8K6c GiSbRk kw f WeY JXd RBB xy qjEV} Applying the curl operator to \eqref{EQ01}, we obtain the vorticity equation   \begin{align}   \begin{split}   \omega_t + \Lambda^{\alpha}\omega + u \cdot \nabla \omega = \partial_{1} \rho   .   \end{split}   \label{EQ05}   \end{align} Define $\tilde \Lambda = (I - \Delta)^{1/2}$ and set  \begin{equation} \zeta= \omega -  \SS\rho , \llabel{EQ06} \end{equation} where   \begin{equation}   \SS   = \partial_{1}\tilde \Lambda^{-\alpha}   = \partial_{1}(I-\Delta)^{-\alpha/2}   .   \llabel{EQ07}   \end{equation} The equation satisfied by $\zeta$ is obtained by replacing $\omega$ with $\zeta+\SS\rho$  in \eqref{EQ05} and combining the resulting equation with \eqref{EQ02}. We get   \begin{align}   \begin{split}   \zeta_{t}   +   \Lambda^{\alpha} \zeta   +   u\cdot \nabla \zeta   &=   -\SS \rho_t    -   u\cdot \nabla \SS \rho   -   \Lambda^{\alpha} \SS \rho   +   \partial_1 \rho\\   &=   [\SS, u\cdot \nabla]\rho    -   (\tilde \Lambda^{-\alpha} \Lambda^{\alpha} - I)\partial_1 \rho   .   \end{split}   \label{EQ08}   \end{align} Therefore, the equation for the generalized vorticity $\zeta$ reads   \begin{align}   \begin{split}   \zeta_{t}   +   \Lambda^{\alpha} \zeta   +   u\cdot \nabla \zeta   =   [\SS, u\cdot \nabla]\rho    -   N \rho   ,   \end{split}   \label{EQ10}   \end{align} where we set    \begin{equation}   N = (\tilde \Lambda^{-\alpha} \Lambda^{\alpha} - I)\partial_1   .   \label{EQ09}   \end{equation} The operator $N$ is a Fourier multiplier with the symbol    \begin{align}   m(\xi) = \frac{|\xi|^{\alpha}\xi_{1}}{(1+|\xi|^2)^{\alpha/2}} - \xi_{1}   .   \llabel{EQ11}   \end{align} It is possible to check that the symbol satisfies the assumptions of  the H\"{o}rmander-Mikhlin theorem and  thus $\Vert N\rho \Vert_{L^\qq} \leq C\Vert \rho \Vert_{L^\qq}$ for $1 < \qq < \infty$. However, as asserted in the next lemma, a stronger statement holds. Namely,  the operator $N$ defined in \eqref{EQ09} is smoothing of order~$1$. \par\heyu{ F5lr 3dFrxG lT c sby AEN cqA 98 1IQ4 UGpB k0gBeJ 6D n 9Jh kne 5f5 18 umOu LnIa spzcRf oC 0 StS y0D F8N Nz F2Up PtNG 50tqKT k2 e 51y Ubr szn Qb eIui Y5qa SGjcXi El 4 5B5 Pny Qtn UO MHis kTC2 KsWkjh a6 l oMf gZK} \cole \begin{Lemma} \label{L01} Consider the Fourier multiplier $T_{\tilde m}$ with the symbol    \begin{equation}   \tilde m(\xi) = (|\xi|^{2}+1)^{1/2} m(\xi)      .   \llabel{EQ12}   \end{equation} Then $T_{\tilde m}$ is a H\"{o}rmander-Mikhlin operator  satisfying   \begin{equation}   \Vert T_{\tilde m}f \Vert_{L^\qq}    \les   \Vert f \Vert_{L^\qq}      \comma    f\in L^\qq   ,   \label{EQ13}   \end{equation} for  $1<\qq<\infty$. \end{Lemma} \colb \par\heyu{ G3n Hp h0gn NQ7q 0QxsQk gQ w Kwy hfP 5qF Ww NaHx SKTA 63ClhG Bg a ruj HnG Kf4 6F QtVt SPgE gTeY6f JG m B3q gXx tR8 RT CPB1 8kQa jtt6GD rK b 1VY LV3 RgW Ir AyZf 69V8 VM7jHO b7 z Lva XTT VI0 ON KMBA HOwO Z7dPky Cg} An equivalent way of stating \eqref{EQ13} is   \begin{equation}   \Vert N f\Vert_{L^\qq}   +   \Vert \nabla N f\Vert_{L^\qq}   \les   \Vert f\Vert_{L^\qq}      \comma f\in L^{\qq}   \commaone \qq\in(1,\infty)   .    \notag   \end{equation} \par\heyu{ U S74 Hln FZM Ha br8m lHbQ NSwwdo mO L 6q5 wvR exV ej vVHk CEdX m3cU54 ju Z SKn g8w cj6 hR 1FnZ Jbkm gKXJgF m5 q Z5S ubX vPK DB OCGf 4srh 1a5FL0 vY f RjJ wUm 2sf Co gRha bxyc 0Rgava Rb k jzl teR GEx bE MMhL} \begin{proof}[Proof of Lemma~\ref{L01}] It suffices to prove that the symbol    \begin{equation}   \tilde m(\xi)   =   \xi_{1}   \frac{ (1+|\xi|^2)^{\alpha/2}-|\xi|^{\alpha}}{(1+|\xi|^2)^{\alpha-1/2}}   \llabel{EQ15}   \end{equation} satisfies the H\"{o}rmander-Mikhlin condition   \begin{equation}   |\partial^{\alpha} \tilde m(\xi)|   \leq   \frac{C(|\alpha|)}{|\xi|^{\alpha}}   \comma   \alpha \in \mathbb N_{0}^{2}   \comma   \xi \in \mathbb R^{2} \setminus \{0\}   .   \llabel{EQ16}   \end{equation} Since $\xi_{1}/(1+|\xi|^2)^{1/2}$ is of H\"{o}rmander-Mikhlin type, it is sufficient to prove that   \begin{equation}   \bar{m}(\xi)     =(1+|\xi|^2)^{1-\alpha/2}       (        (1+|\xi|^2)^{\alpha/2}        -        |\xi|^{\alpha}     )   \llabel{EQ17}   \end{equation} satisfies the H\"{o}rmander-Mikhlin condition. In order to check this, we write   \begin{equation}   \bar{m}(\xi)=\frac{\alpha}{2} \int_{0}^{1}\frac{(1+|\xi|^2)^{1-\alpha/2}}{(t+|\xi|^2)^{1-\alpha/2}}\,dt   \llabel{EQ18}   \end{equation} and then verify that the condition holds  for the low and high frequencies, i.e., when $|\xi|\les 1$ and $|\xi|\gtrsim 1$ respectively. \end{proof}   \par\heyu{ Zbh3 axosCq u7 k Z1P t6Y 8zJ Xt vmvP vAr3 LSWDjb VP N 7eN u20 r8B w2 ivnk zMda 93zWWi UB H wQz ahU iji 2T rXI8 v2HN ShbTKL eK W 83W rQK O4T Zm 57yz oVYZ JytSg2 Wx 4 Yaf THA xS7 ka cIPQ JGYd Dk0531 u2 Q IKf REW} Next, we recall a version of the Kato-Ponce inequality from \cite{KWZ}. \par\heyu{ YcM KM UT7f dT9E kIfUJ3 pM W 59Q LFm u02 YH Jaa2 Er6K SIwTBG DJ Y Zwv fSJ Qby 7f dFWd fT9z U27ws5 oU 5 MUT DJz KFN oj dXRy BaYy bTvnhh 2d V 77o FFl t4H 0R NZjV J5BJ pyIqAO WW c efd R27 nGk jm oEFH janX f1ONEc yt} \cole \begin{Lemma}[\cite{KWZ}] \label{L02} Let $s \in (0,1)$ and $f,g\in \mathcal S(\mathbb R^2)$. For $1<q<\infty$ and $j \in \{1,2 \}$, the inequality   \begin{equation}   \Vert [\Lambda^s \partial_j,g]f\Vert_{L^q}   \leq   C\Vert f\Vert_{L^{q_1}} \Vert\Lambda^{1+s}g\Vert_{L^{\tilde q_{1}}} +    C\Vert \Lambda^s f\Vert_{L^{q_2}} \Vert \Lambda g\Vert_{L^{\tilde q_{2}}}   \llabel{EQ19}   \end{equation} holds, where  $q_1,  \tilde q_1, \tilde q_2 \in [q,\infty]$ and $q_2 \in [q, \infty)$ satisfy $1/q = 1/q_{1} + 1/\tilde q_{1} =1/q_{2} + 1/\tilde q_{2} $ and $C=C(q_1,  \tilde q_{1}, \tilde q_{2},q_2,s)$. \end{Lemma} \colb    \par\heyu{ o INt D90 ONa nd awDR Ki2D JzAqYH GC T B0p zdB a3O ot Pq1Q VFva YNTVz2 sZ J 6ey Ig2 N7P gi lKLF 9Nzc rhuLeC eX w b6c MFE xfl JS E8Ev 9WHg Q1Brp7 RO M ACw vAn ATq GZ Hwkd HA5f bABXo6 EW H soW 6HQ Yvv jc ZgRk} Finally, we recall from \cite{CC,J} an inequality useful for treating the fractional coercive term. \par\heyu{ OWAb VA0zBf Ba W wlI V05 Z6E 2J QjOe HcZG Juq90a c5 J h9h 0rL KfI Ht l8tP rtRd qql8TZ GU g dNy SBH oNr QC sxtg zuGA wHvyNx pM m wKQ uJF Kjt Zr 6Y4H dmrC bnF52g A0 3 28a Vuz Ebp lX Zd7E JEEC 939HQt ha M sup Tcx} \cole \begin{Lemma}[\cite{CC,J}] \label{L03} Consider the operator $\Lambda$ defined in \eqref{EQ04} on $\mathbb R^2$. If $\theta, \Lambda^s \theta \in L^p$, where $p \geq 2$, then   \begin{equation}   \int_{\mathbb R^2}    |\theta|^{p-2}\theta \Lambda^{s}\theta   \,dx    \geq    \frac{2}{p} \int_{\mathbb R^2} (\Lambda^{s/2}(|\theta|^{p/2}))^2     \,dx   ,   \label{EQ20}   \end{equation} for all $s\in(0,2)$. \end{Lemma} \colb \par\heyu{ VaZ 32 pPdb PIj2 x8Azxj YX S q8L sof qmg Sq jm8G 4wUb Q28LuA ab w I0c FWN fGn zp VzsU eHsL 9zoBLl g5 j XQX nR0 giR mC LErq lDIP YeYXdu UJ E 0Bs bkK bjp dc PLie k8NW rIjsfa pH h 4GY vMF bA6 7q yex7 sHgH G3GlW0 y1} \startnewsection{An $L^q$ inequality for the vorticity and a Kato-Ponce type commutator estimate}{sec03}  The following lemma provides an $L^q$ bound for the modified vorticity~$\zeta$.  \par\heyu{ W D35 mIo 5gE Ub Obrb knjg UQyko7 g2 y rEO fov QfA k6 UVDH Gl7G V3LvQm ra d EUO Jpu uzt BB nrme filt 1sGSf5 O0 a w2D c0h RaH Ga lEqI pfgP yNQoLH p2 L AIU p77 Fyg rj C8qB buxB kYX8NT mU v yT7 YnB gv5 K7 vq5N} \par\heyu{ efB5 ye4TMu Cf m E2J F7h gqw I7 dmNx 2CqZ uLFthz Il B 1sj KA8 WGD Kc DKva bk9y p28TFP 0r g 0iA 9CB D36 c8 HLkZ nO2S 6Zoafv LX b 8go pYa 085 EM RbAb QjGt urIXlT E0 G z0t YSV Use Cj DvrQ 2bvf iIJCdf CA c WyI O7m} \cole \begin{Lemma} \label{L04} Assume that $u_0,\rho_0\in W^{s,q}(\mathbb R^{2})$, where $s> 1$ and $q>2$. Then we have   \begin{equation}   \Vert \zeta \Vert_{L^q} \leq Ce^{Ct}   \comma t\ge0   \label{EQ21}   \end{equation} and   \begin{equation}   \Vert \omega \Vert_{L^q} \leq Ce^{Ct}   \comma t\ge0   ,   \label{EQ22}   \end{equation} where $C = C(\Vert \omega_{0} \Vert_{L^q}, \Vert \rho_{0} \Vert_{L^q})$. Moreover, we have   \begin{equation}   \int_{0}^{t} \Vert \Lambda^{\alpha/2} (|\zeta|^{q/2}) \Vert_{L^{2}}^{2}    \,dx   \leq   C e^{C t}   ,   \label{EQ23}   \end{equation} for all $t\ge0$. \end{Lemma} \colb \par\heyu{ lyc s5 Rjio IZt7 qyB7pL 9p y G8X DTz JxH s0 yhVV Ar8Z QRqsZC HH A DFT wvJ HeH OG vLJH uTfN a5j12Z kT v GqO yS8 826 D2 rj7r HDTL N7Ggmt 9M z cyg wxn j4J Je Qb7e MmwR nSuZLU 8q U NDL rdg C70 bh EPgp b7zk 5a32N1 Ib} Above and in the sequel, the exponent $q>2$  and the parameter $s>1$ are considered fixed, so we do not indicate dependence of constants on these parameters. \par\heyu{ J hf8 XvG RmU Fd vIUk wPFb idJPLl NG e 1RQ RsK 2dV NP M7A3 Yhdh B1R6N5 MJ i 5S4 R49 8lw Y9 I8RH xQKL lAk8W3 Ts 7 WFU oNw I9K Wn ztPx rZLv NwZ28E YO n ouf xz6 ip9 aS WnNQ ASri wYC1sO tS q Xzo t8k 4KO z7 8LG6} The main step in the proof of Lemma~\ref{L04} and Theorem \ref{T01} is an inhomogeneous Kato-Ponce type commutator estimate, which is stated next. \par\heyu{ GMNC ExoMh9 wl 5 vbs mnn q6H g6 WToJ un74 JxyNBX yV p vxN B0N 8wy mK 3reR eEzF xbK92x EL s 950 SNg Lmv iR C1bF HjDC ke3Sgt Ud C 4cO Nb4 EF2 4D 1VDB HlWA Tyswjy DO W ibT HqX t3a G6 mkfG JVWv 40lexP nI c y5c kRM} \cole \begin{Lemma} \label{L05} Denote   \begin{equation}   \bSS := |\nabla|(I-\Delta)^{-\alpha/2}.   \llabel{EQ200}   \end{equation} Then, for  $j\in\{1,2\}$ and $0\le\mu\le \alpha$, we have   \begin{equation}   \Vert [\Lambda^{\mu} \SS\partial_{j},g]f\Vert_{L^q}   \leq   C   \Vert \nabla g\Vert_{L^{r_1}} \Vert\Lambda^{\mu}\bSS f\Vert_{L^{\tilde r_{1}}}    +   C    \Vert \Lambda^{\mu+1}\bSS g\Vert_{L^{r_2}} \Vert f \Vert_{L^{\tilde r_{2}}}   ,   \llabel{EQ199}   \end{equation} where $r_1,  \tilde r_1, \tilde r_2 \in [q,\infty]$ and $r_2 \in [q, \infty)$ satisfy $1/q = 1/r_{1} + 1/\tilde r_{1} =1/r_{2} + 1/\tilde r_{2} $ and where $C=C(r_1,  \tilde r_{1}, \tilde r_{2},r_2,q)$. \end{Lemma} \colb \par\heyu{ D3o wV BdxQ m6Cv LaAgxi Jt E sSl ZFw DoY P2 nRYb CdXR z5HboV TU 8 NPg NVi WeX GV QZ7b jOy1 LRy9fa j9 n 2iE 1S0 mci 0Y D3Hg UxzL atb92M hC p ZKL JqH TSF RM n3KV kpcF LUcF0X 66 i vdq 01c Vqk oQ qu1u 2Cpi p5EV7A gM} \par\heyu{ O Rcf ZjL x7L cv 9lXn 6rS8 WeK3zT LD P B61 JVW wMi KE uUZZ 4qiK 1iQ8N0 83 2 TS4 eLW 4ze Uy onzT Sofn a74RQV Ki u 9W3 kEa 3gH 8x diOh AcHs IQCsEt 0Q i 2IH w9v q9r NP lh1y 3wOR qrJcxU 4i 5 5ZH TOo GP0 zE qlB3} \begin{proof}[Proof of Lemma~\ref{L05} (sketch)]  We follow the strategy from \cite{KP} (cf.~also \cite{KWZ}) and consider the commutator in three regions defined by the supports of $\Phi_k$ below. Namely, we write   \begin{align}    \begin{split}    [\Lambda^{\mu}\SS\partial_{j}, g]f    &=    c_{0}\sum_{k=1}^{3}\iint e^{ix(\xi+\eta)} \left(\frac{|\xi + \eta|^{\mu}(\xi_{1} + \eta_{1})(\xi_{j} + \eta_{j})}{(1+|\xi + \eta|^{2})^{\alpha/2}} - \frac{|\xi|^{\mu}\xi_{1}\xi_{j}}{(1+|\xi|^{2})^{\alpha/2}}\right)     \\&      \indeq\indeq\indeq\indeq\indeq\indeq   \indeqtimes   \hat{f}(\xi)\hat{g}(\eta) \Phi_{k}\left(\frac{|\xi|}{|\eta|}\right)\,d\eta\,d\xi   \\&=   c_{0}\sum_{k=1}^{3}\iint e^{ix(\xi+\eta)} A_{k}(\xi, \eta)\,d\eta\,d\xi   ,   \label{EQ56}   \end{split}   \end{align} where $\Phi_{k}\colon\mathbb R \rightarrow [0,1]$ are $C^{\infty}$ cut-off functions such that   \begin{equation}   \sum_{k=1}^{3}\Phi_{k} =1   \inon{on $[0,\infty)$}   \llabel{EQ57}   \end{equation} with   \begin{equation}   \supp{\Phi_{1}} \subseteq [-1/2,1/2]   \comma   \supp{\Phi_{2}} \subseteq [1/4,3]   \comma   \supp{\Phi_{3}} \subseteq [2,\infty]   \llabel{EQ58}   \end{equation} and   \begin{equation}   A_{k}(\xi, \eta)=\left(\frac{|\xi + \eta|^{\mu}(\xi_{1} + \eta_{1})(\xi_{j} + \eta_{j})}{(1+|\xi + \eta|^{2})^{\alpha/2}} - \frac{|\xi|^{\mu}\xi_{1}\xi_{j}}{(1+|\xi|^{2})^{\alpha/2}}\right)\hat{f}(\xi)\hat{g}(\eta)\Phi_{k}\left(\frac{|\xi|}{|\eta|}\right)   .   \llabel{EQ59}   \end{equation} Thus, the commutator \eqref{EQ56} may be rewritten as    \begin{align}   \begin{split}   [\Lambda^{s-1}\SS\partial_{j}, u_j]\rho   =   \sum_{k=1}^{3}\iint e^{ix(\xi+\eta)} A_{k}(\xi, \eta)\,d\eta\,d\xi   .   \llabel{EQ60}   \end{split}   \end{align} We write $A_{1}$ as    \begin{align}   \begin{split}   A_{1}(\xi, \eta)   &=   \left(\frac{|\xi + \eta|^{\mu}(\xi_{1} + \eta_{1})(\xi_{j} + \eta_{j})(1+|\eta|^2)^{\alpha/2}}{(1+|\xi + \eta|^{2})^{\alpha/2}|\eta|^{\mu+2}}    -   \frac{|\xi|^{\mu}\xi_{1}\xi_{j}(1+|\eta|^2)^{\alpha/2}}{(1+|\xi|^{2})^{\alpha/2}|\eta|^{\mu+2}}\right)   \\&\indeq      \times\hat{f}(\xi) (\Lambda^{\mu+1}\bSS g)\hat~(\eta) \Phi_{1}\left(\frac{|\xi|}{|\eta|}\right)   \\&=   \sigma_{1}(\xi, \eta) \hat{f}(\xi) (\Lambda^{\mu +1}\bSS g)\hat~(\eta)   .   \llabel{EQ62}   \end{split}   \end{align} It is elementary to show that   \begin{align}   \begin{split}   |\sigma_{1}|   \leq   C   ,   \llabel{EQ63}   \end{split}   \end{align} as well as more generally   \begin{align}   \begin{split}   |\partial^{\alpha}\partial^{\beta}\sigma_{1}|   \leq   \frac{       C(|\alpha|, |\beta|)   }{   (|\xi|+|\eta|)^{|\alpha|+|\beta|}   }   \comma   \alpha,\beta \in \mathbb N_{0}^{2}   .   \llabel{EQ64}   \end{split}   \end{align} By the Coifman-Meyer theorem, we get   \begin{align}   \begin{split}   \left\Vert  \iint e^{ix(\xi+\eta)} A_{1}(\xi, \eta)\,d\eta\,d\xi \right\Vert_{L^q}   \les   \Vert  f \Vert_{L^{r_{1}}}   \Vert \Lambda^{\mu+1}\bSS g \Vert_{L^{\tilde r_{1}}}   ,   \llabel{EQ65}   \end{split}   \end{align} where $1/q=1/r_{1}+1/\tilde r_{1}$. For $A_{3}$, we write    \begin{align}   \begin{split}   A_{3}(\xi, \eta)   &=   \left(\frac{|\xi + \eta|^{\mu}(\xi_{1} + \eta_{1})(\xi_{j} + \eta_{j})}{(1+|\xi + \eta|^{2})^{\alpha/2}} - \frac{|\xi|^{\mu}\xi_{1}\xi_{j}}{(1+|\xi|^{2})^{\alpha/2}}\right)\hat{f}(\xi)\hat{g}(\eta)\Phi_{3}\left(\frac{|\xi|}{|\eta|}\right)   \\&=   \frac{(1+|\xi|^2)^{\alpha/2}}{|\eta||\xi|^{\mu+1}}   \left(\frac{|\xi + \eta|^{\mu}(\xi_{1} + \eta_{1})(\xi_{j} + \eta_{j})}{(1+|\xi + \eta|^{2})^{\alpha/2}} - \frac{|\xi|^{\mu}\xi_{1}\xi_{j}}{(1+|\xi|^{2})^{\alpha/2}}\right)    (\Lambda^{\mu}\bSS f)\hat~(\xi) (\nabla g)\hat~(\eta)\Phi_{3}\left(\frac{|\xi|}{|\eta|}\right)    \\&=    \sigma_{3}(\xi,\eta)(\Lambda^{\mu}\bSS f)\hat~(\xi) (\nabla g)\hat~(\eta)\Phi_{3}\left(\frac{|\xi|}{|\eta|}\right)   .   \llabel{EQ66}   \end{split}   \end{align} Setting   \begin{align}   \begin{split}   \phi(t)   =   \frac{|\xi + t\eta|^{\mu}(\xi_{1} + t\eta_{1})(\xi_{j} + t\eta_{j})}{(1+|\xi + t\eta|^{2})^{\alpha/2}}      \comma t\in[0,1]     \llabel{EQ666}     \end{split}    \end{align} we have   \begin{align}   \begin{split}   \phi'(t)   &=   \frac{\mu|\xi + t\eta|^{\mu-2}(\xi + t\eta)\eta(\xi_{1} + t\eta_{1})(\xi_{j} + t\eta_{j})}{(1+|\xi + t\eta|^{2})^{\alpha/2}}   +   \frac{|\xi + t\eta|^{\mu}\eta_{1}(\xi_{j} +t \eta_{j})}{(1+|\xi + t\eta|^{2})^{\alpha/2}}   \\&\indeq   +   \frac{|\xi + t\eta|^{\mu}\eta_{j}(\xi_{1} +t \eta_{1})}{(1+|\xi + t\eta|^{2})^{\alpha/2}}   +   \frac{\alpha|\xi + t\eta|^{\mu}(\xi_{1} + t\eta_{1})(\xi_{j} + t\eta_{j})(\xi + t\eta)\eta}{(1+|\xi + t\eta|^{2})^{\alpha/2+1}}   .   \llabel{EQ667}   \end{split}   \end{align} Note that in the region $\Phi_{3}>0$, we have $|\xi|\geq 2|\eta|$. Therefore,   \begin{align}   \begin{split}   |\sigma_{3}|   \leq   C   ,   \llabel{EQ668}   \end{split}   \end{align} as well as more generally   \begin{align}   \begin{split}   |\partial^{\alpha}\partial^{\beta}\sigma_{3}|   \leq   \frac{     C(|\alpha|, |\beta|)   }{     (|\xi|+|\eta|)^{|\alpha|+|\beta|}   }   \comma   \alpha,\beta \in \mathbb N_{0}^{2}   .   \llabel{EQ669}   \end{split}   \end{align} By the Coifman-Meyer theorem, we get   \begin{align}   \begin{split}   \left\Vert  \iint e^{ix(\xi+\eta)} A_{3}(\xi, \eta)\,d\eta\,d\xi \right\Vert_{L^q}   \les   \Vert  \nabla g \Vert_{L^{q_1}}   \Vert \Lambda^{\mu}\bSS f\Vert_{L^{q_{2}}}   ,   \llabel{EQ301}   \end{split}   \end{align} where $1/q=1/q_1+1/q_2$. For $A_{2}$, we use the complex interpolation inequality. Since the argument is the same as in \cite{KP}, we omit the proof. By combining the estimates for $A_1$, $A_{2}$, and $A_{3}$, we get   \begin{align}   \begin{split}    \Vert   [\Lambda^{\mu}\SS\partial_{j}, g]f   \Vert_{L^{q}}   \les   \Vert \nabla g\Vert_{L^{q_{1}}}\Vert \Lambda^{\mu}\bSS  f\Vert_{L^{q_{2}}}   + \Vert \Lambda^{\mu}\bSS \nabla g\Vert_{L^{q_{3}}} \Vert f \Vert_{L^{q_4}}   ,   \end{split}   \llabel{EQ69}   \end{align} where the parameters $q_1, q_2, q_3, q_4\in [q,\infty]$  satisfy $1/q = 1/q_{1} + 1/q_{2} =1/q_{3} + 1/q_{4} $  and  the implicit constant depends on  $q_1$, $q_2$, $q_3$, $q_4$, and $\mu$. \end{proof} \par\heyu{ lkwG GRn7TO oK f GZu 5Bc zGK Fe oyIB tjNb 8xfQEK du O nJV OZh 8PU Va RonX BkIj BT9WWo r7 A 3Wf XxA 2f2 Vl XZS1 Ttsa b4n6R3 BK X 0XJ Tml kVt cW TMCs iFVy jfcrze Jk 5 MBx wR7 zzV On jlLz Uz5u LeqWjD ul 7 OnY ICG} \begin{proof}[Proof of Lemma~\ref{L04}] Since $s>1$, we have $W^{s,q}({\mathbb R}^2\subseteq L^\infty({\mathbb R}^2)$, and thus   \begin{equation}   \rho_0   \in   L^{\qq}   \comma    \qq\in [q,\infty]   .   \llabel{EQ24}   \end{equation} Using the $L^{\qq}$ conservation property for the density equation \eqref{EQ02}, we get   \begin{equation}   \Vert \rho(t)\Vert_{L^{\qq}}   \le    \Vert \rho_0\Vert_{L^{\qq}}   \les 1   \comma    \qq\in [q,\infty]   ,   \label{EQ25}   \end{equation} where we assume that all the constants depend on $\Vert \rho_0\Vert_{L^q}$ and $\Vert \omega_0\Vert_{L^{q}}$. In order to estimate $\Vert \zeta \Vert_{L^q}$, we multiply the equation \eqref{EQ10} with $|\zeta|^{q-2}\zeta$ and integrate obtaining   \begin{align}   \begin{split}   &\frac{1}{q} \frac{d}{dt}\Vert \zeta \Vert_{L^q}^{q} + \int (\Lambda^{\alpha}\zeta)  | \zeta |^{q-2}\zeta\,dx   \\&\indeq   =   -\int N \rho | \zeta |^{q-2}\zeta\,dx    +   \int [\SS, u\cdot \nabla]\rho | \zeta |^{q-2}\zeta\,dx   =   I_1 + I_2   .                  \label{EQ26}   \end{split}   \end{align} For $I_1$, we have    \begin{equation}   I_1    \leq   \Vert N\rho \Vert_{L^q}    \Vert |\zeta|^{q-2}\zeta \Vert_{L^{q/(q-1)}}   \les    \Vert \rho \Vert_{L^q}    \Vert \zeta \Vert_{L^q}^{q-1}   \les   \Vert \zeta \Vert_{L^q}^{q-1}   ,   \label{EQ27}   \end{equation} where we used H\"{o}lder's inequality and \eqref{EQ25}. Since $u$ is divergence-free, we may rewrite the  commutator as   \begin{align}   \begin{split}   [\SS, u \cdot \nabla]\rho   =   \SS u_{j}\partial_{j}\rho   -   u_{j}\partial_{j}\SS\rho   =   (\partial_{j}\SS)(u_{j}\rho)   -   u_{j}(\partial_{j}\SS)\rho   =   [\partial_{j} \SS, u_{j}]\rho   .   \end{split}   \llabel{EQ28}   \end{align} Observe that $\partial_{j}\SS$ is an operator of order $2-\alpha$. Thus, by Lemma~\ref{L05} with $\mu=0$, we have   \begin{align}   \begin{split}   I_2   & \leq    \Vert [\SS, u\cdot \nabla]\rho\Vert_{L^{q}} \Vert \zeta \Vert_{L^{q}}^{q-1}   =    \Vert [\partial_{j}\SS, u_j]\rho\Vert_{L^{q}} \Vert \zeta \Vert_{L^{q}}^{q-1}   \\& \les   (\Vert \bSS \rho \Vert_{L^{a_1}}\Vert \nabla u \Vert_{L^{b_1}}+\Vert \rho \Vert_{L^{a_2}} \Vert \bSS \nabla u\Vert_{L^{b_2}} )\Vert \zeta \Vert_{L^{q}}^{q-1}   \\&   \les    (\Vert \bSS \rho \Vert_{L^{a_1}}   \Vert \omega \Vert_{L^{b_1}}   +\Vert \rho \Vert_{L^{a_2}}    \Vert \bSS \omega \Vert_{L^{b_2}} )   \Vert \zeta \Vert_{L^{q}}^{q-1}   ,   \end{split}   \llabel{EQ29}   \end{align} with the Lebesgue exponents above satisfying $1/q = 1/a_{1} + 1/b_{1} =1/a_{2} + 1/b_{2}$ and  $a_{1},a_{2},b_{1},b_{2} \in (q,\infty)$. Therefore, choosing $a_1=a_2=q/(\alpha-1)$ and $b_1=b_2=q/(2-\alpha)$,   \begin{align}   \begin{split}   I_2   & \les   (   \Vert \bSS \rho \Vert_{L^{q/(\alpha-1)}}   \Vert \omega \Vert_{L^{q/(2-\alpha)}}   +\Vert \rho \Vert_{L^{q/(\alpha-1)}}    \Vert \bSS \omega \Vert_{L^{q/(2-\alpha)}} )   \Vert \zeta \Vert_{L^{q}}^{q-1}   .   \end{split}   \llabel{EQ30}   \end{align} Now, by the fractional Gagliardo-Nirenberg inequality applied to $|\zeta|^{q/2}$, we have   \begin{equation}   \Vert \zeta\Vert_{L^{r}}   \les   \Vert \zeta\Vert_{L^{q}}^{(r\alpha-2r+2q)/\alpha r}   \Vert \Lambda^{\alpha/2}(|\zeta|^{q/2})\Vert_{L^2}^{4(r-q)/\alpha r q}   \comma    q\leq r\leq 2q/(2-\alpha)   ,   \label{EQ31}   \end{equation} from where   \begin{equation}   \Vert \zeta\Vert_{L^{q/(2-\alpha)}}   \les   \Vert \zeta\Vert_{L^{q}}^{(2-\alpha)/\alpha}   \Vert \Lambda^{\alpha/2}(|\zeta|^{q/2})\Vert_{L^2}^{4(\alpha-1)/\alpha q}   .   \llabel{EQ32}   \end{equation} Also, using the triangle inequality   \begin{align}   \begin{split}   \Vert \omega \Vert_{L^{q/(2-\alpha)}}   &   \leq   \Vert \zeta \Vert_{L^{q/(2-\alpha)}}   +   \Vert \bSS\rho \Vert_{L^{q/(2-\alpha)}}   \les   \Vert \zeta \Vert_{L^{q/(2-\alpha)}}   +   \Vert \rho \Vert_{L^{q/(2-\alpha)}}   \les   \Vert \zeta \Vert_{L^{q/(2-\alpha)}}   +   1   \\&   \les   \Vert \zeta\Vert_{L^{q}}^{(2-\alpha)/\alpha}   \Vert \Lambda^{\alpha/2}(|\zeta|^{q/2})\Vert_{L^2}^{4(\alpha-1)/\alpha q}   +   1   \end{split}   \notag   \llabel{EQ34}   \end{align} we get   \begin{align}   \begin{split}   I_2   &\les   (   \Vert \omega \Vert_{L^{q/(2-\alpha)}}   +   \Vert \bSS \omega \Vert_{L^{q/(2-\alpha)}} )   \Vert \zeta \Vert_{L^{q}}^{q-1}   \\&   \les   \Vert \omega\Vert_{L^{q/(2-\alpha)}}\Vert \zeta\Vert_{L^{q}}^{q-1}   \\&   \les   \Vert \zeta\Vert_{L^{q}}^{(2-\alpha)/\alpha+q-1}   \Vert \Lambda^{\alpha/2}(|\zeta|^{q/2})\Vert_{L^2}^{4(\alpha-1)/\alpha q}   +   \Vert \zeta \Vert_{L^{q}}^{q-1}   .   \end{split}   \label{EQ35}   \end{align} Replacing \eqref{EQ27} and \eqref{EQ35} in \eqref{EQ26} and using \eqref{EQ20} on the coercive term, we obtain   \begin{align}   \begin{split}   &\frac{1}{q} \frac{d}{dt}\Vert \zeta \Vert_{L^q}^{q}    +     \frac{2}{q} \int_{\Omega} (\Lambda^{\alpha/2}(|\zeta|^{q/2}))^2 \,dx   \les   \Vert \zeta \Vert_{L^q}^{q-1}   +   \Vert \zeta\Vert_{L^{q}}^{(2-\alpha)/\alpha+q-1}   \Vert \Lambda^{\alpha/2}(|\zeta|^{q/2})\Vert_{L^2}^{4(\alpha-1)/\alpha q}   .   \llabel{EQ36}   \end{split}   \end{align} Since $4(\alpha-1)/\alpha q<2$, we may use Young's inequality with exponents $\alpha q/(\alpha q-2\alpha+2)$ and $\alpha q/2(\alpha-1)$  to get   \begin{align}   \begin{split}   & \frac{d}{dt}\Vert \zeta \Vert_{L^q}^{q}    +     \int_{\Omega} (\Lambda^{\alpha/2}(|\zeta|^{q/2}))^2 \,dx   \les   \Vert \zeta \Vert_{L^q}^{q-1}   +   \Vert \zeta \Vert_{L^q}^{q}   ,   \end{split}   \label{EQ37}   \end{align} where the implicit constant depends on the initial data. The inequality \eqref{EQ21} then follows by applying the Gronwall inequality, while \eqref{EQ22} is a consequence of \eqref{EQ21} and the triangle inequality. Finally, \eqref{EQ23} holds by using \eqref{EQ21} in \eqref{EQ37} and integrating. \end{proof} \par\heyu{ G9i Ry bTsY JXfr Rnub3p 16 J BQd 0zQ OkK ZK 6DeV gpXR ceOExL Y3 W KrX YyI e7d qM qanC CTjF W71LQ8 9m Q w1g Asw nYS Me WlHz 7ud7 xBwxF3 m8 u sa6 6yr 0nS ds Ywuq wXdD 0fRjFp eL O e0r csI uMG rS OqRE W5pl ybq3rF rk} It is important that we may bootstrap the above statement and obtain the conclusion   on the behavior of the $L^{\bar q}$ norm of $\zeta$, and thus of $\omega$, for all $\bar q>q$. \par\heyu{ 7 YmL URU SSV YG ruD6 ksnL XBkvVS 2q 0 ljM PpI L27 Qd ZMUP baOo Lqt3bh n6 R X9h PAd QRp 9P I4fB kJ8u ILIArp Tl 4 E6j rUY wuF Xi FYaD VvrD b2zVpv Gg 6 zFY ojS bMB hr 4pW8 OwDN Uao2mh DT S cei 90K rsm wa BnNU} \cole \begin{Lemma} \label{L06} Assume that $u_0,\rho_0\in W^{s,q}(\mathbb R^2)$, where $s\ge 1$ and $q\in(2,\infty)$. Then for every $\bar q\in (q,\infty)$ and $t_0>0$ we have   \begin{equation}   \Vert \zeta \Vert_{L^{\bar q}} \leq Ce^{Ct}   \comma t\ge t_0   \llabel{EQ38}   \end{equation} and   \begin{equation}   \Vert \omega \Vert_{L^{\bar q}} \leq Ce^{Ct}   \comma t\ge t_0   ,   \label{EQ39}   \end{equation} where $C = C(\Vert \omega_{0} \Vert_{L^q}, \Vert \rho_{0} \Vert_{L^q}, \bar q, t_0)$. Moreover, we have   \begin{equation}   \int_{0}^{t} \Vert \Lambda^{\alpha/2} (|\zeta^{\bar q/2}|) \Vert_{L^{2}}^{2} \,ds   \leq   C e^{C t}   ,   \llabel{EQ40}   \end{equation} for all $t\ge0$ where $C = C(\Vert \omega_{0} \Vert_{L^q}, \Vert \rho_{0} \Vert_{L^q}, \bar q, t_0)$. \end{Lemma} \colb \par\heyu{ sHe6 RpIq1h XF N Pm0 iVs nGk bC Jr8V megl 416tU2 nn o llO tcF UM7 c4 GC8C lasl J0N8Xf Cu R aR2 sYe fjV ri JNj1 f2ty vqJyQN X1 F YmT l5N 17t kb BTPu F471 AH0Fo7 1R E ILJ p4V sqi WT TtkA d5Rk kJH3Ri RN K ePe sR0} {\begin{proof}[Proof of Lemma~\ref{L05}] We first prove that the statement holds for all $\bar q\in [q,2q/(2-\alpha)]$, and the rest follows by an iteration argument. Using \eqref{EQ23} with $t=t_0=1$,  we obtain   \begin{equation}   \left|   \left\{   t\in(0,t_0]:   \Vert \Lambda^{\alpha/2} (|\zeta^{\bar q/2}(t)|) \Vert_{L^{2}}^{2}    \le C   \right\}   \right|   \ge \frac{1}{C}   \llabel{EQ41}   \end{equation} for $C>0$ sufficiently large. It is easy to deduce then that there exists $\bar t\in (0,t_0)$ such that   \begin{equation}   \Vert \Lambda^{\alpha/2} (|\zeta^{\bar q/2}|) (\bar t) \Vert_{L^{2}}   \le C   .   \llabel{EQ42}   \end{equation}   Since also   \begin{equation}   \Vert \zeta(\bar t)\Vert_{L^q}   \leq   C   ,   \llabel{EQ43}   \end{equation} we get by \eqref{EQ31}   \begin{equation}   \Vert \zeta(\bar t)\Vert_{L^{\bar q}}   \leq   C   \llabel{EQ44}   \end{equation} since $q\leq \bar q\leq 2q/(2-\alpha)$. Applying Lemma~\ref{L04} but with $q$ replaced with $\bar q$, we obtain the statement for $\bar q$ in this range. Continuing by induction, we get then the conclusion for all $\bar q\in[q,\infty)$, and the lemma is established. \end{proof} \par\heyu{ xqF qn QjGU IniV gLGCl2 He 7 kmq hEV 4PF dC dGpE P9nB mcvZ0p LY G idf n65 qEu Df Mz2v cq4D MzN6mB FR t QP0 yDD Fxj uZ iZPE 3Jj4 hVc2zr rc R OnF PeO P1p Zg nsHA MRK4 ETNF23 Kt f Gem 2kr 5gf 5u 8Ncu wfJC av6SvQ 2n} \par\heyu{ 1 8P8 RcI kmM SD 0wrV R1PY x7kEkZ Js J 7Wb 6XI WDE 0U nqtZ PAqE ETS3Eq NN f 38D Ek6 NhX V9 c3se vM32 WACSj3 eN X uq9 GhP OPC hd 7v1T 6gqR inehWk 8w L oaa wHV vbU 49 02yO bCT6 zm2aNf 8x U wPO ilr R3v 8R cNWE} \par\heyu{ k7Ev IAI8ok PA Y xPi UlZ 4mw zs Jo6r uPmY N6tylD Ee e oTm lBK mnV uB B7Hn U7qK n353Sn dt o L82 gDi fcm jL hHx3 gi0a kymhua FT z RnM ibF GU5 W5 x651 0NKi 85u8JT LY c bfO Mn0 auD 0t vNHw SAWz E3HWcY TI d 2Hh XML} \par\heyu{ iGi yk AjHC nRX4 uJJlct Q3 y Loq i9j u7K j8 4EFU 49ud eA93xZ fZ C BW4 bSK pyc f6 nncm vnhK b0HjuK Wp 6 b88 pGC 3U7 km CO1e Y8jv Ebu59z mG Z sZh 93N wvJ Yb kEgD pJBj gQeQUH 9k C az6 ZGp cpg rH r79I eQvT Idp35m wW} \startnewsection{The Sobolev persistence for $1< s \leq \alpha$}{sec04}  In this section, we prove our main result, Theorem~\ref{T01}, in the case when $s\leq \alpha$. \par\heyu{ m afR gjD vXS 7a FgmN IWmj vopqUu xF r BYm oa4 5jq kR gTBP PKLg oMLjiw IZ 2 I4F 91C 6x9 ae W7Tq 9CeM 62kef7 MU b ovx Wyx gID cL 8Xsz u2pZ TcbjaK 0f K zEy znV 0WF Yx bFOZ JYzB CXtQ4u xU 9 6Tn N0C GBh WE FZr6} \begin{proof}[Proof of Theorem~\ref{T01} for $s\leq \alpha$] For $j=1,2$, we multiply the $j$-th velocity equation of \eqref{EQ01} with $|u_j|^{q-2}u_j$, integrate the resulting equation with respect to $x$, and sum for $j=1,2$ obtaining   \begin{align}   \begin{split}   &\frac{1}{q}     \frac{d}{dt}    \sum_{j=1}^2   \Vert u_j \Vert_{L^q}^{q}    +     \sum_{j=1}^{2}        \int (\Lambda^{\alpha}u_j) |u_j|^{q-2}u_j \,dx   =   -   \sum_{j=1}^{2}   \int \partial_{j} \pi \cdot |u_j|^{q-2}u_j \,dx   +   \sum_{j=1}^{2}   \int \rho e_{2} \cdot |u_j|^{q-2}u_j \,dx   \end{split}   \llabel{EQ45}   \end{align} since due to the divergence-free condition for $u$ we have $\int (u\cdot \nabla u_j)  |u_j|^{q-2}u_j \,dx=0$ for $j=1,2$. By Lemma~\ref{L03} and H\"{o}lder's inequality, we get   \begin{align}   \begin{split}   &\frac{1}{q}     \frac{d}{dt}    \sum_{j=1}^2   \Vert u_j \Vert_{L^q}^{q}    +  \frac{2}{q}\sum_{j=1}^{2} \Vert \Lambda^{\alpha/2} (|u_j|^{q/2})\Vert_{L^2}^2   \les   \Vert \nabla \pi \Vert_{L^q}\Vert u \Vert_{L^q}^{q-1}     +   \Vert u \Vert_{L^q}^{q-1}     \end{split}   \label{EQ46}   \end{align} where, as above, $q$ is considered fixed (i.e., the constants are allowed to depend on $q$). Using the Calder\'{o}n-Zygmund and Sobolev embedding theorems, we obtain   \begin{align}   \begin{split}   \Vert \nabla \pi \Vert_{L^q}   \les   \Vert  u \Vert_{L^{2q}}     \Vert \omega \Vert_{L^{2q}}    \les   \Vert u\Vert_{L^q}^{1-1/q}   \Vert \omega\Vert_{L^q}^{1/q}   \Vert \omega \Vert_{L^{2q}}    \les   C e^{Ct}   \Vert u\Vert_{L^q}^{1-1/q}   ,   \end{split}   \label{EQ47}   \end{align} where we also used Lemma~\ref{L06} in the last step. Applying \eqref{EQ47}  on the first term of the right hand side of \eqref{EQ46}  gives   \begin{align}   \begin{split}   &    \frac{d}{dt}    \sum_{j=1}^2   \Vert u_j \Vert_{L^q}^{q}    +     \sum_{j=1}^{2} \Vert \Lambda^{\alpha/2} (|u_j|^{q/2})\Vert_{L^2}^2   \les   e^{Ct}   \Vert u\Vert_{L^q}^{q-1/q}   +   \Vert u \Vert_{L^q}^{q-1}   \end{split}    \llabel{EQ48}   \end{align} and thus   \begin{equation}    \Vert u\Vert_{L^q}    \les    e^{C t}    \llabel{EQ68}   \end{equation} by the Gronwall inequality. \par\heyu{ 0rIg w2f9x0 fW 3 kUB 4AO fct vL 5I0A NOLd w7h8zK 12 S TKy 2Zd ewo XY PZLV Vvtr aCxAJm N7 M rmI arJ tfT dd DWE9 At6m hMPCVN UO O SZY tGk Pvx ps GeRg uDvt WTHMHf 3V y r6W 3xv cpi 0z 2wfw Q1DL 1wHedT qX l yoj GIQ} Next, we consider the $L^q$ norm of higher order derivatives. Applying $\Lambda^{s-1}$ to  \eqref{EQ10}, multiplying the resulting equation by  $|\Lambda^{s-1}\zeta |^{q-2}\Lambda^{s-1}\zeta$, and  integrating, we get   \begin{align}   \begin{split}   &\frac{1}{q}\frac{d}{dt}\Vert \Lambda^{s-1}\zeta \Vert_{L^q}^q    +   \int \Lambda^{\alpha}(\Lambda^{s-1}\zeta)| \Lambda^{s-1}\zeta |^{q-2}\Lambda^{s-1}\zeta \,dx           \\&\indeq   =   -   \int \Lambda^{s-1} (u \cdot \nabla \zeta) | \Lambda^{s-1}\zeta |^{q-2}\Lambda^{s-1}\zeta\,dx   +   \int \Lambda^{s-1}( [\SS, u \cdot \nabla]\rho) |\Lambda^{s-1}\zeta|^{q-2}\Lambda^{s-1} \zeta\,dx   \\&\indeq\indeq   -   \int \Lambda^{s-1} N \rho | \Lambda^{s-1}\zeta |^{q-2}\Lambda^{s-1}\zeta \,dx   \\&\indeq   =J_1 + J_2 + J_3   .   \end{split}   \label{EQ49}   \end{align} By  Lemma~\ref{L02},  we estimate   \begin{align}   \begin{split}   J_1   &=   -\int \bigl(\Lambda^{s-1} (u \cdot \nabla \zeta) - u\cdot \Lambda^{s-1}\nabla \zeta \bigr)|    \Lambda^{s-1}\zeta |^{q-2}\Lambda^{s-1}\zeta \,dx   \\&   \leq    \Vert \Lambda^{s-1} (u \cdot \nabla \zeta) - u\cdot \Lambda^{s-1}\nabla \zeta\Vert_{L^{q}} \Vert \Lambda^{s-1}\zeta \Vert_{L^q}^{q-1}     \\&    \leq   \sum_{j=1}^{2}   \Vert \partial_{j}\Lambda^{s-1} (u_j \zeta) - u_j \partial_{j} \Lambda^{s-1} \zeta\Vert_{L^{q}} \Vert \Lambda^{s-1}\zeta \Vert_{L^q}^{q-1}     ,   \end{split}   \label{EQ50}   \end{align} where we used the divergence-free condition and the triangle inequality in the last step. Therefore,   \begin{align}   \begin{split}   J_1   &\les   (     \Vert \zeta \Vert_{L^{r_1}}      \Vert \Lambda^{s-1} \omega \Vert_{L^{r_2}}     +     \Vert \Lambda u \Vert_{L^{r_{3}}}     \Vert \Lambda^{s-1}\zeta \Vert_{L^{r_4}}    ) \Vert \Lambda^{s-1}\zeta \Vert_{L^q}^{q-1}   \\& \les   \bigl(     \Vert \zeta \Vert_{L^{r_1}}( \Vert \Lambda^{s-1} \zeta \Vert_{L^{r_2}} + \Vert \Lambda^{s-1} \SS \rho \Vert_{L^{r_2}})       +     \Vert \omega \Vert_{L^{r_{3}}}     \Vert \Lambda^{s-1}\zeta \Vert_{L^{r_{4}}}   \bigr)    \Vert \Lambda^{s-1}\zeta \Vert_{L^q}^{q-1}   ,   \end{split}    \llabel{EQ135}   \end{align} for any   $r_1, r_2, r_3, r_4 \in (q,\infty)$  such that  $1/q=1/r_{1} + 1/r_{2} = 1/r_{3} + 1/r_{4}$. Choose $r_{1}=r_{2}=r_{3}=r_{4}=2q$ and note that   \begin{equation}    \Vert \Lambda^{s-1} \SS \rho \Vert_{L^{2q}}    \les    \Vert  \rho \Vert_{L^{2q}}    \les 1    \llabel{EQ95}   \end{equation} by $s\leq \alpha$. Therefore, using Lemma~\ref{L06},    \begin{align}   \begin{split}   J_1   \les   e^{Ct}   (\Vert \Lambda^{s-1}\zeta \Vert_{L^{2q}} + 1) \Vert \Lambda^{s-1}\zeta \Vert_{L^q}^{q-1}   .   \end{split}   \llabel{EQ300}   \end{align} By \eqref{EQ31},  have   \begin{equation}   \Vert \zeta\Vert_{L^{2q}}   \les   \Vert \zeta\Vert_{L^{q}}^{(\alpha-1)/\alpha}   \Vert \Lambda^{\alpha/2}(|\zeta|^{q/2})\Vert_{L^2}^{2/\alpha q}   ,    \label{EQ136}   \end{equation} and thus we obtain   \begin{align}   \begin{split}   J_1   & \les   e^{Ct}\Vert \Lambda^{s-1}\zeta \Vert_{L^q}^{(\alpha-1)/\alpha}   \Vert\Lambda^{\alpha/2}(|\Lambda^{s-1}\zeta|)^{q/2} \Vert_{L^2}^{2/\alpha q}   +   e^{Ct} \Vert \Lambda^{s-1}\zeta \Vert_{L^q}^{q-1}   .   \end{split}   \llabel{EQ51}   \end{align} For $J_2$,  we write   \begin{align}   \begin{split}   \Lambda^{s-1}\bigl([\SS, u \cdot \nabla]\rho \bigr)   &=   \Lambda^{s-1}\bigl(\SS \bigl((u \cdot \nabla)\rho \bigr)\bigr)   -   \Lambda^{s-1}\bigl((u\cdot\nabla)\SS\rho\bigr)   \\&   =   \Lambda^{s-1}\Bigl(\SS \bigl((u \cdot \nabla)\rho \bigr)\Bigr)   -   u\cdot\nabla\bigl(\Lambda^{s-1} (\SS\rho)\bigr)   \\&\indeq   +   u\cdot\nabla\bigl(\Lambda^{s-1} (\SS\rho)\bigr)   -   \Lambda^{s-1}\bigl((u\cdot\nabla)\SS\rho\bigr)   \\&   =   \Lambda^{s-1}\SS \partial_{j}\bigl(u_j \rho \bigr)   -   u_j\partial_{j}\bigl(\Lambda^{s-1} (\SS\rho)\bigr)   \\&\indeq   +   u_j\partial_{j}\bigl(\Lambda^{s-1} (\SS\rho)\bigr)   -   \Lambda^{s-1}\partial_{j}\bigl(u_j\SS\rho\bigr)   ,   \end{split}   \label{EQ52}   \end{align} where we used the divergence-free condition  \eqref{EQ03} in the last step. The first two  and the last two terms  on the far right side of \eqref{EQ52} form commutators, as we may write   \begin{align}    \begin{split}    \Lambda^{s-1}\bigl([\SS, u \cdot \nabla]\rho \bigr)    &=    [\Lambda^{s-1}\SS\partial_{j}, u_j]\rho    - [\Lambda^{s-1}\partial_{j},u_j] \SS\rho    .   \end{split}   \label{EQ53}   \end{align} For the second commutator in \eqref{EQ53}, we apply Lemma~\ref{L02} and obtain   \begin{align}   \begin{split}    \Vert    [\Lambda^{s-1}\partial_{j},u_j] \SS\rho    \Vert_{L^{q}}    \les    \Vert \SS\rho\Vert_{L^{p_{1}}}\Vert \Lambda^{s}u\Vert_{L^{p_{2}}}     + \Vert \Lambda^{s-1}\SS\rho \Vert_{L^{p_{3}}}\                                         \Vert \nabla u \Vert_{L^{p_4}}    ,   \end{split}   \llabel{EQ54}   \end{align} where $1/q =1/p_{1} + 1/p_{2} =1/p_{3} + 1/p_{4}  $ and $p_{i}  \in (q,\infty)$ for $i=1,2,3,4$. Thus, by Lemma~\ref{L05},    \begin{align}   \begin{split}   J_2   & \leq    \Vert \Lambda^{s-1}[\SS, u \cdot \nabla]\rho \Vert_{L^q} \Vert \Lambda^{s-1}\zeta \Vert_{L^q}^{q-1}\\   &\les   (\Vert \nabla u\Vert_{L^{2q}}\Vert \Lambda^{s-1}\bSS \rho\Vert_{L^{2q}} + \Vert \Lambda^{s-1}\bSS\nabla u\Vert_{L^{2q}} \Vert \rho \Vert_{L^{2q}})\Vert \Lambda^{s-1}\zeta\Vert_{L^q}^{q-1}   \\&\indeq   +   (\Vert \SS\rho\Vert_{L^{2q}}\Vert \Lambda^{s}u\Vert_{L^{2q}} + \Vert \rho \Vert_{L^{2q}}\Vert \SS\Lambda^{s}u \Vert_{L^{2q}})\Vert \Lambda^{s-1}\zeta\Vert_{L^q}^{q-1}\\   &=   J_{21} + J_{22}   .   \end{split}   \label{EQ70}   \end{align} Now, we use the conservation property \eqref{EQ25} for the density and the fact that the  operator $\Lambda^{s-1}\bSS$ is of H\"ormander-Mikhlin type, and we get   \begin{align}   \begin{split}   J_{21}   &\les   \Vert \nabla u\Vert_{L^{2q}}   \Vert \Lambda^{s-1}\zeta\Vert_{L^q}^{q-1}   \les   \Vert \omega\Vert_{L^{2q}}   \Vert \Lambda^{s-1}\zeta\Vert_{L^q}^{q-1}   \les   e^{Ct} \Vert \Lambda^{s-1}\zeta\Vert_{L^q}^{q-1}   ,   \end{split}   \label{EQ71}   \end{align} where we applied Lemma~\ref{L04} in the last step. For $J_{22}$,  we choose $p_2 =p_4= 2q$. Then by the conservation of density and \eqref{EQ136} we have   \begin{align}   \begin{split}   J_{22}   &\les   (\Vert \Lambda^{s-1}\omega \Vert_{L^{2q}}   + \Vert \omega\Vert_{L^{2q}} )   \Vert \Lambda^{s-1}\zeta\Vert_{L^q}^{q-1}   \\&   \les   \bigl(   \Vert \Lambda^{s-1}\zeta\Vert_{L^{2q}}   + \Vert \Lambda^{s-1} \SS \rho \Vert_{L^{2q}}   +  e^{C t}   \bigr)   \Vert \Lambda^{s-1}\zeta\Vert_{L^q}^{q-1}   \\&   \les   \bigl(   \Vert \Lambda^{s-1}\zeta\Vert_{L^{2q}}   + e^{C t}   \bigr)   \Vert \Lambda^{s-1}\zeta\Vert_{L^q}^{q-1}   \\&   \les   \Vert \Lambda^{s-1}\zeta\Vert_{L^{q}}^{(\alpha-1)/\alpha+ q-1}   \Vert   \Lambda^{\alpha/2}(|\Lambda^{s-1}\zeta|^{q/2})\Vert^{2/\alpha q}_{L^2}   +   e^{Ct}\Vert \Lambda^{s-1}\zeta\Vert_{L^q}^{q-1}   .   \end{split}   \label{EQ72}   \end{align} From \eqref{EQ70}, \eqref{EQ71}, and \eqref{EQ72}, we conclude   \begin{equation}     J_2     \leq       \frac{1}{C}          \Vert   \Lambda^{\alpha/2}(|\Lambda^{s-1}\zeta|^{q/2})\Vert^{2}_{L^2}       +       C       \Vert \Lambda^{s-1}\zeta\Vert_{L^q}^{q}       +       C       e^{Ct} \Vert \Lambda^{s-1}\zeta\Vert_{L^q}^{q-1}    .    \llabel{EQ137}   \end{equation} For $J_3$, we use Lemma~\ref{L01} and obtain   \begin{align}   J_3   &\leq   \Vert \Lambda^{s-1} N \rho \Vert_{L^q}\Vert \Lambda^{s-1} \zeta\Vert_{L^q}^{q-1}   \les   \Vert  \rho \Vert_{L^q}\Vert \Lambda^{s-1} \zeta\Vert_{L^q}^{q-1}   \les   \Vert \Lambda^{s-1} \zeta\Vert_{L^q}^{q-1}   .   \llabel{EQ73}   \end{align} Combining the estimates of $J_1$, $J_2$, and $J_3$, using Young's inequality, we get   \begin{align}   \begin{split}   & \frac{d}{dt}\Vert \Lambda^{s-1}\zeta\Vert_{L^q}^{q}   + \frac{1}{C} \Vert \Lambda^{\alpha/2}( |\Lambda^{s-1}\zeta|^{q/2}) \Vert_{L^2}^2        \les   e^{Ct}\Vert \Lambda^{s-1}\zeta\Vert_{L^q}^{q-1}   + e^{C t}   .   \end{split}   \label{EQ74}   \end{align} Setting   \begin{align}   X   &=   \Vert \Lambda^{s-1} \zeta \Vert_{L^q}^q,   \llabel{EQ75}\\   \bar X   &=   \Vert \Lambda^{\alpha/2}(|\Lambda^{s-1}\zeta|^{q/2})\Vert_{L^2}^2   ,   \llabel{EQ76}   \end{align} we may rewrite \eqref{EQ74} as   \begin{align}   \frac{d}{dt} X   +   \frac{1}{C}{\bar X}    \les   e^{Ct} (X^{1-1/q}+1)   .   \llabel{EQ78}   \end{align} Therefore, by the Gronwall lemma,   \begin{align}   \Vert \Lambda^{s-1} \zeta \Vert_{L^q}   \les   e^{C t}   \comma t\ge0    .   \llabel{EQ79}   \end{align} Similarly to Lemma~\ref{L06}, we also obtain   \begin{align}   \Vert \Lambda^{s-1} \zeta \Vert_{L^{\bar q}}   \les   e^{Ct}   \comma t\ge0   ,   \llabel{EQ799}   \end{align} for all $\bar q \in [q, \infty)$, where the constant $C$ depends on $\bar q$. Consequently, we get   \begin{align}   \Vert \Lambda^{s-1} \omega \Vert_{L^{\bar q}}   \leq   \Vert \Lambda^{s-1} \zeta \Vert_{L^{\bar q}}   +   \Vert \Lambda^{s-1} \SS\rho \Vert_{L^{\bar q}}   \les   e^{C t}    \comma \bar q\in [q,\infty)    .   \label{EQ80}   \end{align} \par\heyu{ AdE EK v7Ta k7cA ilRfvr lm 8 2Nj Ng9 KDS vN oQiN hng2 tnBSVw d8 P 4o3 oLq rzP NH ZmkQ Itfj 61TcOQ PJ b lsB Yq3 Nul Nf rCon Z6kZ 2VbZ0p sQ A aUC iMa oRp FW fviT xmey zmc5Qs El 1 PNO Z4x otc iI nwc6 IFbp wsMeXx y8} Next, we consider the evolution of $\Vert \Lambda^s \rho \Vert_{L^q}$. We apply $\Lambda^s$ to the equation \eqref{EQ02},  multiply it by $| \Lambda^s \rho |^{q-2} \Lambda^{s}\rho$, and integrate obtaining   \begin{equation}    \frac{1}{q}\frac{d}{dt}\Vert \Lambda^s \rho \Vert_{L^q}^q     +    \int \Lambda^s (u \cdot \nabla \rho) | \Lambda^s \rho |^{q-2}\Lambda^s \rho \,dx    = 0    .   \llabel{EQ81}   \end{equation} Therefore, using Lemma~\ref{L02},   \begin{align}   \begin{split}   \frac{1}{q}\frac{d}{dt}\Vert \Lambda^s \rho \Vert_{L^q}^q   & = -\int \Lambda^s (u \cdot \nabla \rho) | \Lambda^s \rho |^{q-2} \Lambda^{s}\rho \,dx\\   & = -\int \left(\Lambda^s (u \cdot \nabla \rho) - u \cdot \Lambda^s \nabla \rho \right) | \Lambda^s \rho |^{q-2} \Lambda^{s}\rho \,dx\\   &\les   \Vert \Lambda^s (u \cdot \nabla \rho) - u \cdot \Lambda^s \nabla \rho \Vert_{L^{q}} \Vert \Lambda^{s} \rho \Vert_{L^{q}}^{q-1}  \\   &\les   \left(\Vert \Lambda^s u \Vert_{L^{s_1}} \Vert \nabla \rho \Vert_{L^{s_{2}}} + \Vert \Lambda u \Vert_{L^{\infty}} \Vert \Lambda^s \rho \Vert_{L^q}\right) \Vert \Lambda^{s} \rho \Vert_{L^{q}}^{q-1}   ,   \end{split}   \llabel{EQ82}   \end{align} under the conditions  $s_{1},s_{2} \in (q, \infty)$ and  $1/q=1/s_{1} + 1/s_{2}$. Now, choose   \begin{equation}    s_2= q+ \frac{s-1}{C_0}    ,    \llabel{EQ138}   \end{equation} where $C_0$ is a positive  constant.  Note that $s_1,s_2\in(q,\infty)$. If $C_0$ is sufficiently large, we may use the fractional Gagliardo-Nirenberg inequality to write   \begin{equation}    \Vert \nabla \rho\Vert_{L^{s_2}}    \les    \Vert \rho \Vert_{L^{q}}^{1-\lambda}    \Vert \Lambda^s \rho \Vert_{L^{q}}^{\lambda}      \llabel{EQ89}   \end{equation} with $\lambda\in(0,1)$.  Therefore, using \eqref{EQ80},   \begin{align}   \begin{split}   \frac{1}{q}\frac{d}{dt}\Vert \Lambda^s \rho \Vert_{L^q}^q   &\les   \left(\Vert \Lambda^{s-1} \omega \Vert_{L^{s_1}}    \Vert \rho \Vert_{L^{q}}^{1-\lambda}\Vert \Lambda^s \rho \Vert_{L^{q}}^{\lambda}     +   \Vert \Lambda u \Vert_{L^{\infty}} \Vert \Lambda^s \rho \Vert_{L^q}\right) \Vert \Lambda^{s} \rho \Vert_{L^{q}}^{q-1}   \\&\les   e^{C t}   \Vert \Lambda^s \rho \Vert_{L^{q}}^{q+\lambda-1}     +   \Vert \Lambda u \Vert_{L^{\infty}} \Vert \Lambda^s \rho \Vert_{L^q}^{q}    .   \end{split}   \label{EQ83}   \end{align} Let  $\bar q\in[q,\infty)$ be sufficiently large so that we have   \begin{align}   \begin{split}   \Vert \Lambda u \Vert_{L^{\infty}}   &\les   \Vert \Lambda u\Vert_{L^{\bar q}}^{1-\mu}   \Vert \Lambda^{s-1} (\Lambda u) \Vert_{L^{\bar q}}^{\mu}   \end{split}   \llabel{EQ84}   \end{align} where $\mu\in(0,1)$. Then we get   \begin{align}   \begin{split}   \Vert \Lambda u \Vert_{L^{\infty}}   &\les   \Vert \Lambda u\Vert_{L^{\bar q}}^{1-\mu}   \Vert \Lambda^{s-1} (\Lambda u) \Vert_{L^{\bar q}}^{\mu}   \les   \Vert \omega\Vert_{L^{\bar q}}^{1-\mu}   \Vert \Lambda^{s-1} \omega \Vert_{L^{\bar q}}^{\mu}   \les   e^{Ct}   \end{split}   \llabel{EQ84a}   \end{align} by \eqref{EQ39} and \eqref{EQ80}. Hence, continuing from \eqref{EQ83}, we get   \begin{align}   \begin{split}   \frac{d}{dt}\Vert \Lambda^s \rho \Vert_{L^q}^q   &\les   e^{Ct}   (      1+ \Vert \Lambda^s \rho \Vert_{L^{q}}^{q}     )    .   \llabel{EQ86}   \end{split}   \end{align} The proof of persistence for $s\in(1,\alpha]$ is then concluded by an application of the Gronwall lemma. \par\heyu{ l J4A 6OV 0qR zr St3P MbvR gOS5ob ka F U9p OdM Pdj Fz 1KRX RKDV UjveW3 d9 s hi3 jzK BTq Zk eSXq bzbo WTc5yR RM o BYQ PCa eZ2 3H Wk9x fdxJ YxHYuN MN G Y4X LVZ oPU Qx JAli DHOK ycMAcT pG H Ikt jlI V25 YY oRC7} It remains to prove the uniqueness of solutions. Consider two solutions  $(u^{(1)}, p^{(1)}, \rho^{(1)})$ and $(u^{(2)}, p^{(2)}, \rho^{(2})$ of the system  \eqref{EQ01}--\eqref{EQ03}, and set   \begin{align}   \begin{split}   & U   = u^{(1)} - u^{(2)}   \\    & R = \rho^{(1)} - \rho^{(2)}   \\    & P= p^{(1)}- p^{(2)}   .   \end{split}   \llabel{EQ134}   \end{align}  Subtracting the equations for $(u^{(1)}, p^{(1)}, \rho^{(1)})$ and $(u^{(2)}, p^{(2)}, \rho^{(2})$, we get   \begin{align}      U_t      +\Lambda^{\alpha} U     + U\cdot \nabla u^{(1)}      +u^{(2)}\cdot \nabla U      +     \nabla P     =     Re_{2}     \label{EQ155}   \end{align}  and   \begin{align}      R_t + U\cdot \nabla \rho^{(1)} + u^{(2)}\cdot \nabla R =0     .    \label{EQ156}   \end{align}  We shall establish uniqueness in the space $L^2({\mathbb R}^2)\times L^{r}({\mathbb R}^2)$ where   \begin{equation}    r=\frac{4q}{2q-q\alpha+4}    .    \llabel{EQ141}   \end{equation} Note that  $1<r<\infty$ and  $(U(0),R(0))=(0,0)\in L^2({\mathbb R}^2) \times L^{r}({\mathbb R}^2)$. From \eqref{EQ155}, we get   \begin{align}    \begin{split}    &    \frac{1}{2}\frac{d}{dt} \Vert U\Vert_{L^2}^2    + \Vert \Lambda^{\alpha/2} U\Vert_{L^2}^2    \les    \Vert \nabla u^{(1)}\Vert_{L^\infty} \Vert U\Vert_{L^2}^2    +    \Vert U\Vert_{L^{r/(r-1)}}    \Vert R\Vert_{L^{r}}    \\&\indeq    \les    \Vert \nabla u^{(1)}\Vert_{L^\infty} \Vert U\Vert_{L^2}^2    +    \Vert \Lambda^{\alpha/2}U\Vert_{L^{2}}    \Vert R\Vert_{L^{r}}    +    \Vert U\Vert_{L^{2}}    \Vert R\Vert_{L^{r}}    ,    \end{split}    \label{EQ90}   \end{align} where we used   \begin{equation}    \frac{r}{r-1} \leq \frac{4}{2-\alpha}    ,    \llabel{EQ145}   \end{equation} which follows from $r\ge 4/(2+\alpha)$  and this holds by $q\alpha\ge 2$. Also, \eqref{EQ156} implies   \begin{align}    \begin{split}    \frac{d}{dt} \Vert R\Vert_{L^r}^r    \les    \Vert U\Vert_{L^{4/(2-\alpha)}}    \Vert \nabla \rho^{(1)}\Vert_{L^{q}}    \Vert R\Vert_{L^{r}}^{r-1}    ,    \end{split}    \llabel{EQ140}   \end{align} from where   \begin{align}    \begin{split}    \frac{d}{dt} \Vert R\Vert_{L^r}^2    \les    \Vert U\Vert_{L^{4/(2-\alpha)}}    \Vert \nabla \rho^{(1)}\Vert_{L^{q}}    \Vert R\Vert_{L^{r}}    \les    \Vert \Lambda^{\alpha/2}U\Vert_{L^{2}}    \Vert \rho^{(1)}\Vert_{W^{s,q}}     \Vert R\Vert_{L^{r}}    .    \end{split}    \label{EQ142}   \end{align} Now, $u^{(1)}\in L_{\rm loc}^{\infty}((0,\infty)([0,T], W^{s,\bar q}({\mathbb R}^2))$ for all $\bar q\in [q,\infty)$, and $W^{s,\bar q}({\mathbb R}^2)\subseteq W^{1,\infty}({\mathbb R}^2)$ for all $\bar q$ sufficiently large. Thus \eqref{EQ90} and \eqref{EQ142} imply $U(t)=0$ and $R(t)=0$ for all $t\ge0$. \end{proof} \par\heyu{ 4thS sJClD7 6y x M6B Rhg fS0 UH 4wXV F0x1 M6Ibem sT K SWl sG9 pk9 5k ZSdH U31c 5BpQeF x5 z a7h WPl LjD Yd KH1p OkMo 1Tvhxx z5 F LLu 71D UNe UX tDFC 7CZ2 473sjE Re b aYt 2sE pV9 wD J8RG UqQm boXwJn HK F Mps XBv} \begin{Remark} \label{R01} {\rm Note that the identity \eqref{EQ53} only uses the additivity  of $\Lambda^{s-1}$ and the fact that it commutes with the differential operators. Thus, for any multiplier operator $T$, we have   \begin{align}    \begin{split}    T\bigl([ \SS, u \cdot \nabla]\rho \bigr)    &=    [T \SS\partial_{j}, u_j]\rho    - [T\partial_{j},u_j]  \SS\rho    .   \end{split}   \notag   \end{align} The proof  of this identity uses the fact that $u$ is divergence-free. } \end{Remark} \par\heyu{ AsX 8N YRZM wmZQ ctltsq of i 8wx n6I W8j c6 8ANB wz8f 4gWowk mZ P Wlw fKp M1f pd o0yT RIKH MDgTl3 BU B Wr6 vHU zFZ bq xnwK kdmJ 3lXzIw kw 7 Jku JcC kgv FZ 3lSo 0ljV Ku9Syb y4 6 zDj M6R XZI DP pHqE fkHt 9SVnVt Wd} \par\heyu{ y YNw dmM m7S Pw mqhO 6FX8 tzwYaM vj z pBS NJ1 z36 89 00v2 i4y2 wQjZhw wF U jq0 UNm k8J 8d OOG3 QlDz p8AWpr uu 4 D9V Rlp VVz QQ g1ca Eqev P0sFPH cw t KI3 Z6n Y79 iQ abga 0i9m RVGbvl TA g V6P UV8 Eup PQ 6xvG} \par\heyu{ bcn7 dQjV7C kw 5 7NP WUy 9Xn wF 9ele bZ8U YJDx3x CB Y CId PCE 2D8 eP 90u4 9NY9 Jxx9RI 4F e a0Q Cjs 5TL od JFph ykcz Bwoe97 Po h Tql 1LM s37 cK hsHO 5jZx qpkHtL bF D nvf Txj iyk LV hpwM qobq DM9A0f 1n 4 i5S Bc6} \startnewsection{The  Sobolev persistence for $s>\alpha$}{sec05}  We now consider the persistence of regularity when $s>\alpha$. \par\heyu{ trq VX wgQB EgH8 lISLPL O5 2 EUv i1m yxk nL 0RBe bO2Y Ww8Jhf o1 l HlU Mie sst dW w4aS WrYv Osn5Wn 3w f wzH RHx Fg0 hK FuNV hjzX bg56HJ 9V t Uwa lOX fT8 oi FY1C sUCg CETCIv LR 0 AgT hCs 9Ta Zl 6ver 8hRt edkAUr kI} \par\heyu{ n Sbc I8n yEj Zs VOSz tBbh 7WjBgf aA F t4J 6CT UCU 54 3rba vpOM yelWYW hV B RGo w5J Rh2 nM fUco BkBX UQ7UlO 5r Y fHD Mce Wou 3R oFWt baKh 70oHBZ n7 u nRp Rh3 SIp p0 Btqk 5vhX CU9BHJ Fx 7 qPx B55 a7R kO yHmS} \par\heyu{ h5vw rDqt0n F7 t oPJ UGq HfY 5u At5k QLP6 ppnRjM Hk 3 HGq Z0O Bug FF xSnA SHBI 7agVfq wf g aAl eH9 DMn XQ QTAA QM8q z9trz8 6V R 2gO MMV uMg f6 tGLZ WEKq vkMEOg Uz M xgN 4Cb Q8f WY 9Tk7 3Gg9 0jy9dJ bO v ddV Zmq} \par\heyu{ Jjb 5q Q5BS Ffl2 tNPRC8 6t I 0PI dLD UqX KO 1ulg XjPV lfDFkF h4 2 W0j wkk H8d xI kjy6 GDge M9mbTY tU S 4lt yAV uor 6w 7Inw Ch6G G9Km3Y oz b uVq tsX TNZ aq mwkz oKxE 9O0QBQ Xh x N5L qr6 x7S xm vRwT SBGJ Y5uo5w SN} \begin{proof}[Proof of Theorem~\ref{T01} for the case $s>\alpha$] Let $J_1$, $J_2$, and $J_3$ be as in  \eqref{EQ49}. For $J_1$,  \eqref{EQ50} and Lemma~\ref{L02} imply   \begin{align}   \begin{split}   J_1   & \les   (\Vert \Lambda^{s-1}\zeta \Vert_{L^{r_2}}\Vert \Lambda u \Vert_{L^{r_{1}}}     +\Vert \zeta \Vert_{L^{r_1}} \Vert \Lambda^{s-1} \omega \Vert_{L^{r_2}} ) \Vert \Lambda^{s-1}\zeta \Vert_{L^q}^{q-1}   \\& \les   \left(\Vert \Lambda^{s-1}\zeta \Vert_{L^{r_{2}}}\Vert \omega \Vert_{L^{r_{1}}}   +\Vert \zeta \Vert_{L^{r_1}} \Vert \Lambda^{s-1} \zeta \Vert_{L^{r_2}}    + \Vert \zeta \Vert_{L^{r_1}}\Vert \Lambda^{s-1} \SS \rho \Vert_{L^{r_2}} \right) \Vert \Lambda^{s-1}\zeta \Vert_{L^q}^{q-1}   \\& \les   e^{Ct}   ( \Vert \Lambda^{s-1} \zeta \Vert_{L^{r_2}} + \Vert \Lambda^{s-\alpha}  \rho \Vert_{L^{r_2}} ) \Vert \Lambda^{s-1}\zeta \Vert_{L^q}^{q-1}   ,   \end{split}   \label{EQ92}   \end{align} for any  $r_1, r_2\in (q,\infty)$ such that $1/q=1/r_{1} + 1/r_{2} $.  We restrict   \begin{equation}    r_2 \in \left(q,  2q/(2-\alpha)\right)    \llabel{EQ112}   \end{equation} so that we may use the inequality \eqref{EQ31} obtaining   \begin{equation}    \Vert \Lambda^{s-1}\zeta\Vert_{L^{r_2}}    \les    \Vert \Lambda^{s-1}\zeta\Vert_{L^{q}}^{1-\theta_1}    \Vert \Lambda^{\alpha/2}(|\Lambda^{s-1}\zeta|^{q/2})\Vert_{L^2}^{2\theta_1/ q}    \label{EQ119}   \end{equation} where $\theta_1=2(r_2-q)/\alpha r_2$. Also, we have   \begin{align}     \Vert \Lambda^{s-\alpha}\rho \Vert_{L^{r_{2}}}     \les     \Vert \rho \Vert_{L^{q}}^{1-\theta_2}\Vert \Lambda^{s}\rho \Vert^{\theta_2}_{L^q}     \les     \Vert \Lambda^{s}\rho \Vert^{\theta_2}_{L^q}   \label{EQ93}   \end{align} with $\theta_2=(2/q-2/r_2+s-\alpha)/s$. Thus, by \eqref{EQ92} and \eqref{EQ119}, we obtain   \begin{align}    \begin{split}    J_{1}    &\les     e^{Ct}    \Vert \Lambda^{s-1}\zeta \Vert_{L^q}^{1-\theta_1}    \Vert\Lambda^{\alpha/2}(|\Lambda^{s-1}\zeta|^{q/2}) \Vert_{L^2}^{2\theta_1/q}    +    e^{Ct}    \Vert \Lambda^{s}\rho \Vert_{L^{q}}^{\theta_2}\Vert \Lambda^{s-1} \zeta \Vert_{L^{q}}^{q-1}    .    \llabel{EQ96}    \end{split}    \end{align} The term $J_2$ is rewritten using \eqref{EQ53} as   \begin{align}   \begin{split}    J_2    &=    \int        [\Lambda^{s-1}\SS\partial_{j}, u_j]\rho    |\Lambda^{s-1}\zeta|^{q-2}\Lambda^{s-1} \zeta\,dx    -    \int     [\Lambda^{s-1}\partial_{j},u_j] \SS\rho    |\Lambda^{s-1}\zeta|^{q-2}\Lambda^{s-1} \zeta\,dx    \\&    = J_{21} + J_{22}    .   \end{split}    \llabel{EQ118}   \end{align} For the first term, we have   \begin{align}   \begin{split}   J_{21}   &\les   (\Vert \nabla u\Vert_{L^{r_1}}   \Vert \Lambda^{s-1}\bSS\rho\Vert_{L^{r_2}}   +    \Vert \Lambda^{s-1}\bSS\nabla u\Vert_{L^{r_3}}   \Vert \rho\Vert_{L^{r_{4}}} )   \Vert \Lambda^{s-1}\zeta\Vert_{L^q}^{q-1}   \\& \les   e^{C t}   (   \Vert \Lambda^{s-\alpha}\rho\Vert_{L^{r_2}}   +    \Vert \Lambda^{s-\alpha}\omega\Vert_{L^{r_3}})   \Vert \Lambda^{s-1}\zeta\Vert_{L^q}^{q-1}   ,   \end{split}   \label{EQ97}   \end{align} where $r_3, r_4\in(q,\infty)$ are such that $1/r_3+1/r_4=1/q$. For $\Vert \Lambda^{s-\alpha}\rho\Vert_{L^{r_2}}$, we  use \eqref{EQ93}, while for $\Vert \Lambda^{s-\alpha}\omega\Vert_{L^{r_3}}$, we have by the triangle inequality   \begin{align}   \begin{split}   \Vert \Lambda^{s-\alpha}\omega\Vert_{L^{r_{3}}}    &\les   \Vert \Lambda^{s-\alpha}\zeta\Vert_{L^{r_{3}}}   +   \Vert \Lambda^{s-\alpha}\SS\rho\Vert_{L^{{r_{3}}}}   \les   \Vert \zeta\Vert_{L^{q}}^{1-\theta_3}   \Vert \Lambda^{s-1}\zeta\Vert_{L^{q}}^{\theta_3}   +   \Vert \Lambda^{(s-2\alpha+1)_{+}}\rho\Vert_{L^{r_3}}   \\&   \les               e^{Ct}   \Vert \Lambda^{s-1}\zeta\Vert_{L^{q}}^{\theta_3}   +   \Vert \Lambda^{(s-2\alpha+1)_{+}}\rho\Vert_{L^{r_3}}   ,   \llabel{EQ99}   \end{split}   \end{align} where  $\theta_3=(2/q-2/r_3+s-\alpha)/(s-1)$, as long as $r_3$ is sufficiently close to $q$. From \eqref{EQ97} we thus obtain   \begin{align}    \begin{split}    J_{21}    &\le    e^{Ct}    \left(     \Vert \Lambda^{s}\rho \Vert^{\theta_2}_{L^q}      +     \Vert \Lambda^{s-1}\zeta\Vert_{L^{q}}^{\theta_3}      +      1      \colb    \right)    \Vert \Lambda^{s-1}\zeta\Vert_{L^q}^{q-1}    \end{split}    \label{EQ98}   \end{align} if $s \leq 2\alpha-1$, and   \begin{align}    \begin{split}    J_{21}    &\le    e^{Ct}    \left(     \Vert \Lambda^{s}\rho \Vert^{\theta_2}_{L^q}      +     \Vert \Lambda^{s-1}\zeta\Vert_{L^{q}}^{\theta_3}      +     \Vert \Lambda^{s}\rho\Vert_{L^q}^{(s-2\alpha+1)/s+\epsilon_0}    \right)    \Vert \Lambda^{s-1}\zeta\Vert_{L^q}^{q-1}    \comma s> 2\alpha-1    ,    \end{split}    \llabel{EQ98a}   \end{align} where $\epsilon_0>0$ is arbitrarily small if $r_3$ is sufficiently close to $q$. Since $\theta_2 > (s-2\alpha+1)/s$, we obtain that \eqref{EQ98} holds even if $s>2\alpha-1$ as long as $r_3>q$ is sufficiently close to $q$. For $J_{22}$,  we recall \eqref{EQ70}, by which   \begin{align}   \begin{split}    J_{22}     &   \les    \bigl(      \Vert \bSS\rho\Vert_{L^{r_1}}      \Vert \Lambda^{s}u\Vert_{L^{r_2}}       + \Vert \rho \Vert_{L^{\infty}}        \Vert \bSS\Lambda^{s}u \Vert_{L^{q}}     \bigr)      \Vert \Lambda^{s-1}\zeta\Vert_{L^q}^{q-1}     \\&    \les    e^{C t}    (      \Vert \Lambda^{s}u\Vert_{L^{r_2}}       +      \Vert \bSS\Lambda^{s}u \Vert_{L^{q}}    )      \Vert \Lambda^{s-1}\zeta\Vert_{L^q}^{q-1}    \\&    \les    e^{C t}    (      \Vert \Lambda^{s-1}\zeta\Vert_{L^{r_2}}       +     \Vert \Lambda^{s-1} \SS \rho \Vert_{L^{r_2}}      +         \Vert \bSS\Lambda^{s}u \Vert_{L^{q}}    )      \Vert \Lambda^{s-1}\zeta\Vert_{L^q}^{q-1}   \end{split}    \llabel{EQ85}   \end{align} and thus   \begin{align}   \begin{split}    J_{22}     &   \les    e^{C t}    (      \Vert \Lambda^{s-1}\zeta\Vert_{L^{r_2}}       +     \Vert \Lambda^{s-\alpha} \rho \Vert_{L^{r_2}}      +         \Vert \bSS\Lambda^{s}u \Vert_{L^{q}}    )      \Vert \Lambda^{s-1}\zeta\Vert_{L^q}^{q-1}    \\&    \les    e^{C t}    (      \Vert \Lambda^{s-1}\zeta\Vert_{L^{r_2}}       +     \Vert \Lambda^{s-\alpha} \rho \Vert_{L^{r_2}}      +      \Vert \bSS\Lambda^{s-1}\omega \Vert_{L^{q}}    )      \Vert \Lambda^{s-1}\zeta\Vert_{L^q}^{q-1}    \\&    \les    e^{C t}    (      \Vert \Lambda^{s-1}\zeta\Vert_{L^{r_2}}       +     \Vert \Lambda^{s-\alpha} \rho \Vert_{L^{r_2}}      +      \Vert \bSS\Lambda^{s-1}\zeta \Vert_{L^{q}}      +      \Vert \bSS\Lambda^{s-1}\SS\rho \Vert_{L^{q}}    )      \Vert \Lambda^{s-1}\zeta\Vert_{L^q}^{q-1}   .   \end{split}    \llabel{EQ100}   \end{align} Note that the last two terms inside the parentheses are lower order compared to the first two. Therefore,   \begin{align}   \begin{split}    J_{22}     &    \les    e^{C t}    (      \Vert \Lambda^{s-1}\zeta\Vert_{L^{r_2}}       +     \Vert \Lambda^{s-\alpha} \rho \Vert_{L^{r_2}}      +       \Vert \zeta\Vert_{L^{r_2}}       +       \Vert \zeta\Vert_{L^{q}}      +     \Vert \rho \Vert_{L^{r_2}}          +     \Vert \rho \Vert_{L^{q}}         )         \Vert \Lambda^{s-1}\zeta\Vert_{L^q}^{q-1}    \\&    \les    e^{C t}    (      \Vert \Lambda^{s-1}\zeta\Vert_{L^{r_2}}       +     \Vert \Lambda^{s-\alpha} \rho \Vert_{L^{r_2}}      +       1    )         \Vert \Lambda^{s-1}\zeta\Vert_{L^q}^{q-1}   .   \end{split}    \llabel{EQ87}   \end{align} The right hand side does not lead to any new terms compared to the estimate for $J_1$ in \eqref{EQ92}, except for the lower order third term inside the parentheses. \par\heyu{ G p3h Ccf QNa fX Wjxe AFyC xUfM8c 0k K kwg psv wVe 4t FsGU IzoW FYfnQA UT 9 xcl Tfi mLC JR XFAm He7V bYOaFB Pj j eF6 xI3 CzO Vv imZ3 2pt5 uveTrh U6 y 8wj wAy IU3 G1 5HMy bdau GckOFn q6 a 5Ha R4D Ooj rN Ajdh} Next, we treat $J_{3}$. When $s\le 2$, we have   \begin{align}   J_3   &\les   \Vert \Lambda^{s-1} N \rho \Vert_{L^q}\Vert \Lambda^{s-1} \zeta\Vert_{L^q}^{q-1}   \les   \Vert \rho \Vert_{L^{q}}   \Vert\Lambda^{s-1} \zeta\Vert_{L^q}^{q-1}   \les   \Vert \Lambda^{s-1} \zeta\Vert_{L^q}^{q-1}   ,   \llabel{EQ107}   \end{align} while if $s\ge 2$,   \begin{align}     \begin{split}   J_3   &\les   \Vert \Lambda^{s-1} N \rho \Vert_{L^q}\Vert \Lambda^{s-1} \zeta\Vert_{L^q}^{q-1}   \les   \Vert \Lambda^{s-2}\rho \Vert_{L^{q}}   \Vert\Lambda^{s-1} \zeta\Vert_{L^q}^{q-1}   \\&   \les   \Vert \rho\Vert_{L^{q}}^{2/s}   \Vert \Lambda^{s}\rho\Vert_{L^{q}}^{(s-2)/s}   \Vert\Lambda^{s-1} \zeta\Vert_{L^q}^{q-1}   \les   \Vert \Lambda^{s}\rho\Vert_{L^{q}}^{(s-2)/s}   \Vert \Lambda^{s-1} \zeta\Vert_{L^q}^{q-1}    .     \end{split}      \llabel{EQ121}   \end{align} We thus conclude   \begin{align}   \begin{split}   &\frac{d}{dt}\Vert \Lambda^{s-1}\zeta \Vert_{L^q}^q    +   \frac{1}{C}   \Vert \Lambda^{\alpha/2} (|\Lambda^{s-1}\zeta|^{q/2})\Vert_{L^2}^2   \\&\indeq   \les      e^{Ct}      \Vert \Lambda^{s-1}\zeta \Vert_{L^q}^{q}         +      e^{Ct}      \Vert \Lambda^{s-1}\zeta \Vert_{L^q}^{q-1}      +      e^{Ct}      \Vert \Lambda^{s}\rho \Vert_{L^{q}}^{\theta_2}\Vert \Lambda^{s-1} \zeta \Vert_{L^{q}}^{q-1}   \\&\indeq\indeq   +   \Vert \Lambda^{s}\rho\Vert_{L^{q}}^{((s-2)/s)_{+}}   \Vert \Lambda^{s-1} \zeta\Vert_{L^q}^{q-1}      .   \end{split}    \label{EQ122}   \end{align} \par\heyu{ SmhO tphQpc 9j X X2u 5rw PHz W0 32fi 2bz1 60Ka4F Dj d 1yV FSM TzS vF 1YkR zdzb YbI0qj KM N XBF tXo CZd j9 jD5A dSrN BdunlT DI a A4U jYS x6D K1 X16i 3yiQ uq4zoo Hv H qNg T2V kWG BV A4qe o8HH 70FflA qT D BKi 461} Next, we consider  $\Vert \Lambda^{s} \rho \Vert_{L^q}$.  First, we have by Sobolev embedding, with  $q^{*}=\max\{2/(\alpha-1),q\}+1$,   \begin{align}   \Vert \Lambda u \Vert_{L^{\infty}}   \les   \Vert \Lambda^{\alpha-1}\Lambda u \Vert_{L^{q^{*}}}    +   \Vert \Lambda u \Vert_{L^{q^{*}}}   \les   \Vert \Lambda^{\alpha-1} \zeta \Vert_{L^{q^{*}}}   +   \Vert \Lambda^{\alpha-1} \SS\rho \Vert_{L^{q^{*}}}    +   \Vert \omega \Vert_{L^{q^{*}}}   \les   \phi(t)   ,   \label{EQ108}   \end{align} where we used Theorem~\ref{T01} in the third inequality and where   \begin{equation}    \phi(t)=C \exp {(C\exp (C t))}    \llabel{EQ94}   \end{equation} with sufficiently large $C$. (The dependence on $t$ can be improved, but we do not optimize the dependence in this paper.) Thus, by Lemma~\ref{L02},    \begin{align}   \begin{split}   \frac{1}{q}\frac{d}{dt}\Vert \Lambda^s \rho \Vert_{L^q}^q   & = -\int \bigl(\Lambda^s (u \cdot \nabla \rho) - u \cdot \Lambda^s \nabla \rho \bigr) | \Lambda^s \rho |^{q-2} \Lambda^{s}\rho \,dx   \\&\les   \Vert \Lambda^s (u \cdot \nabla \rho) - u \cdot \Lambda^s \nabla \rho \Vert_{L^{q}} \Vert \Lambda^{s} \rho \Vert_{L^{q}}^{q-1}     \\&\les  \left(\Vert \Lambda^s u \Vert_{L^{s_1}} \Vert \nabla \rho \Vert_{L^{s_{2}}} + \Vert \Lambda u \Vert_{L^{\infty}} \Vert \Lambda^s \rho \Vert_{L^q}\right) \Vert \Lambda^{s} \rho \Vert_{L^{q}}^{q-1}   \\&\les   \left(\Vert \Lambda^{s-1} \omega \Vert_{L^{s_1}} \Vert \nabla \rho \Vert_{L^{s_{2}}} + \phi(t) \Vert \Lambda^s \rho \Vert_{L^q}\right) \Vert \Lambda^{s} \rho \Vert_{L^{q}}^{q-1}   \end{split}   \label{EQ109}   \end{align} where $s_1,s_2\in(q,\infty)$ are such that $1/s_1+1/s_2=1/q$. At this point, we employ an inequality from \cite{BM}, which gives   \begin{align}    \Vert \nabla\rho\Vert_{L^{s_{2}}}    \les    \Vert \rho\Vert_{L^{{\bar s_2}}}^{1-1/s}    \Vert \Lambda^{s}\rho\Vert_{L^{{q}}}^{1/s}    \les    \Vert \Lambda^{s}\rho\Vert_{L^{{q}}}^{1/s}   \label{EQ110}   \end{align} where $1/s_2=1/sq+(1/\bar s_2)(1-1/s)$, assuming that $    s_2\le  q s $, which is equivalent to   \begin{equation}    s_1\ge \frac{qs}{qs-1}    .    \llabel{EQ125}   \end{equation} From \eqref{EQ109} and \eqref{EQ110} we then obtain   \begin{align}   \begin{split}    \frac{d}{dt}\Vert \Lambda^s \rho \Vert_{L^q}       \les        \Vert \Lambda^{s-1} \omega \Vert_{L^{s_1}}          \Vert \Lambda^{s}\rho \Vert_{L^{q}}^{1/s}             + \phi(t) \Vert \Lambda^s \rho \Vert_{L^q}   .   \end{split}    \label{EQ126} \end{align} If   \begin{equation}    s_1\le \frac{2q}{2-\alpha}    \label{EQ123}   \end{equation} we may apply \eqref{EQ31} and obtain   \begin{align}    \begin{split}    \Vert \Lambda^{s-1}\omega\Vert_{L^{s_{1}}}    &\leq    \Vert \Lambda^{s-1}\zeta\Vert_{L^{s_{1}}}    +    \Vert \Lambda^{s-1}\SS\rho\Vert_{L^{s_{1}}}    \\&    \les    \Vert \Lambda^{s-1}\zeta\Vert_{L^{{ q}}}^{1-\theta_3}\Vert \Lambda^{\alpha/2}(|\Lambda^{s-1}\zeta|^{{q}/2})\Vert^{2\theta_3/{q}}_{L^2}    +    \Vert \Lambda^{s-\alpha}\rho\Vert_{L^{s_{1}}}    \\&    \les    \Vert \Lambda^{s-1}\zeta\Vert_{L^{{ q}}}^{1-\theta_3}\Vert \Lambda^{\alpha/2}(|\Lambda^{s-1}\zeta|^{{q}/2})\Vert^{2\theta_3/{q}}_{L^2}    +     \Vert \rho \Vert_{L^{q}}^{1-\theta_4}\Vert \Lambda^{s}\rho \Vert^{\theta_4}_{L^q}    \\&    \les    \Vert \Lambda^{s-1}\zeta\Vert_{L^{{ q}}}^{1-\theta_3}\Vert \Lambda^{\alpha/2}(|\Lambda^{s-1}\zeta|^{{q}/2})\Vert^{2\theta_3/{q}}_{L^2}    +     \Vert \Lambda^{s}\rho \Vert^{\theta_4}_{L^q}    ,    \llabel{EQ111}   \end{split}   \end{align} where  $\theta_3=2(s_1-q)/\alpha s_1$ and  $\theta_4=(s-\alpha-2/s_1+2/q)/s$. Therefore, by \eqref{EQ126},   \begin{align}   \begin{split}    \frac{d}{dt}\Vert \Lambda^s \rho \Vert_{L^q}     \les    \Vert \Lambda^{s-1}\zeta\Vert_{L^{{ q}}}^{1-\theta_3}\Vert \Lambda^{\alpha/2}(|\Lambda^{s-1}\zeta|^{{q}/2})\Vert^{2\theta_3/{q}}_{L^2}    \Vert \Lambda^{s}\rho \Vert_{L^{q}}^{1/s}    +     \Vert \Lambda^{s}\rho \Vert^{\theta_4+1/s}_{L^q}            + \phi(t) \Vert \Lambda^s \rho \Vert_{L^q}    .   \end{split}    \label{EQ127}   \end{align} Now, in order to conclude the proof, let $\gamma>0$, and denote   \begin{align}   X   &=   \Vert \Lambda^{s-1} \zeta \Vert_{L^q}^{q}   ,   \llabel{EQ113}\\   Y    &=   (\Vert \Lambda^{s} \rho \Vert_{L^q}+1)^{q/\gamma}   ,   \llabel{EQ115}\\   Z   &=   \Vert \Lambda^{\alpha/2}(|\Lambda^{s-1}\zeta|^{q/2})\Vert_{L^2}^2   .   \llabel{EQ116}   \end{align} Then \eqref{EQ122} and \eqref{EQ127} may be rewritten as   \begin{align}   \begin{split}   &\frac{d}{dt} X   +   \frac{1}{C}   Z   \les      e^{Ct}      X      +      e^{Ct}      X^{(q-1)/q}      +      e^{Ct}      Y^{\theta_2\gamma/q}      X^{(q-1)/q}   +   Y^{((s-2)/s)_{+}\gamma/q}   X^{(q-1)/q}   \end{split}    \label{EQ129}   \end{align} and   \begin{align}   \begin{split}    \frac{d}{dt}Y     \les    X^{(1-\theta_3)/q}    Z^{\theta_3/{q}}    Y^{1+\gamma/s q-\gamma/q}    +     Y^{\gamma\theta_4/q+\gamma /sq+1-\gamma/q}            + \phi(t) Y    ,   \end{split}   \label{EQ128}   \end{align} respectively. (We use here that if $(d/dt)\Vert \Lambda^{s}\rho\Vert_{L^q}\leq f$, then $\dot Y\leq f Y^{1-\gamma/q}$.) Adding \eqref{EQ129} and \eqref{EQ128}, we obtain   \begin{align}   \begin{split}   \frac{d}{dt}(X+Y)   +   \frac{1}{C}{Z}   & \les      e^{Ct}      X      +      e^{Ct}      X^{(q-1)/q}      +      e^{Ct}      Y^{\theta_2\gamma/q}      X^{(q-1)/q}   +   Y^{((s-2)/s)_{+}\gamma/q}   X^{(q-1)/q}   \\&\indeq   +    X^{(1-\theta_3)/q}    Z^{\theta_3/{q}}    Y^{1+\gamma/s q-\gamma/q}    +     Y^{\gamma\theta_4/q+\gamma /sq+1-\gamma/q}    + \phi(t)Y   .   \end{split}   \llabel{EQ117}   \end{align} In order to apply the Gronwall lemma, it is sufficient that the conditions   \begin{align}    \begin{split}    &     \frac{\theta_2 \gamma}{q}     + \frac{q-1}{q}    \le 1    \\&    \left(\frac{s-2}{s}\right)_{+}    \frac{\gamma}{q}    + \frac{q-1}{q} \le 1    \\&    \frac{1}{q}    + \frac{\gamma}{s q}     - \frac{\gamma}{q}    \le 0   \\&     \frac{\gamma \theta_4}{q}     + \frac{\gamma }{s q}     - \frac{\gamma}{q}     \le 0    \end{split}    \label{EQ130}   \end{align} hold. The first three conditions may be summarized as   \begin{equation}    \frac{s}{s-1}    \leq    \gamma    \leq    \min\left\{         \frac{1}{\theta_2},          \frac{s}{s-2}        \right\}    \llabel{EQ131}   \end{equation} if $s>2$ and as   \begin{equation}    \frac{s}{s-1}    \leq    \gamma    \leq         \frac{1}{\theta_2}    \llabel{EQ131}   \end{equation} if $s\le 2$. The last condition in \eqref{EQ130} is equivalent to   \begin{equation}    s\geq \frac{1}{1-\theta_4}    .    \llabel{EQ132}   \end{equation} Setting $s_1=q s/(q s-1)$, it is easy to verify that  we may simply take $\gamma=s/(s-1)$ as we have $s/(s-1)\leq 1/\theta_2$. The condition \eqref{EQ123}  can also be checked easily. The proof is concluded by a simple application of a Gronwall lemma. \end{proof} \par\heyu{ GvM gz d7Wr iqtF q24GYc yi f YkW Hv7 EI0 aq 5JKl fNDC NmWom3 Vy X JsN t4W P8y Gg AoAT OkVW Z4ODLt kz a 9Pa dGC GQ2 FC H6EQ ppks xFKMWA fY 0 Jda SYg o7h hG wHtt bb4z 5qrcdc 9C n Amx qY6 m8u Gf 7DZQ 6FBU PPiOxg sQ} \par\heyu{ 0 CZl PYP Ba7 5O iV6t ZOBp fYuNcb j4 V Upb TKX ZRJ f3 6EA0 LDgA dfdOpS bg 1 ynC PUV oRW xe WQMK Smuh 3JHqX1 5A P JJX 2v0 W6l m0 llC8 hlss 1NLWaN hR B Aqf Iuz kx2 sp 01oD rYsR ywFrNb z1 h Gpq 99F wUz lf cQkT} \section*{Acknowledgments}  The authors were supported in part by the NSF grants DMS-1615239 and DMS-1907992. \par\heyu{ sbCv GIIgmf Hh T rM1 ItD gCM zY ttQR jzFx XIgI7F MA p 1kl lwJ sGo dX AT2P goIp 9VonFk wZ V Qif q9C lAQ 4Y BwFR 4nCy RAg84M LJ u nx8 uKT F3F zl GEQt l32y 174wLX Zm 6 2xX 5xG oaC Hv gZFE myDI zj3q10 RZ r ssw ByA} \par\heyu{ 2Wl OA DDDQ Vin8 PTFLGm wi 6 pgR ZQ6 A5T Ll mnFV tNiJ bnUkLy vq 9 zSB P6e JJq 7P 6RFa im6K XPWaxm 6W 7 fM8 3uK D6k Nj 7vhg 4ppZ 4ObMaS aP H 0oq xAB G8v qr qT6Q iRGH BCCN1Z bl T Y4z q8l FqL Ck ghxD UuZw 7MXCD4 ps} 
\small \begin{thebibliography}{[JMWZ]}\bibitem[ACW]{ACW} D.~Adhikari, C.~Cao, H.~Shang, J.~Wu, X.~Xu, and Z.~Ye,  \emph{Global regularity results for the 2{D} {B}oussinesq equations with partial dissipation},   J.~Differential Equations~\textbf{260} (2016), no.~2, 1893--1917. \bibitem[BM]{BM} H.~Brezis and P.~Mironescu,   \emph{Gagliardo-{N}irenberg inequalities and non-inequalities: the full story},   Ann.\ Inst.\ H.~Poincar\'{e} Anal.\ Non Lin\'{e}aire~\textbf{35} (2018), no.~5, 1355--1376. \bibitem[BS]{BS}L.C.~Berselli and S.~Spirito, \emph{On the {B}oussinesq system:  regularity criteria and singular limits}, Methods Appl.\ Anal.~\textbf{18}  (2011), no.~4, 391--416. \bibitem[BrS]{BrS}L.~Brandolese and M.E.~Schonbek,   \emph{Large time decay and growth for solutions of a viscous {B}oussinesq system},   Trans.\ Amer.\ Math.\ Soc.\ ~\textbf{364} (2012), no.~10, 5057--5090. \bibitem[C]{C}D.~Chae, \emph{Global regularity for the 2{D} {B}oussinesq equations with  partial viscosity terms}, Adv.\ Math.~\textbf{203} (2006), no.~2,  497--513.\bibitem[CC]{CC} A.~C\'{o}rdoba and D.~C\'{o}rdoba,   \emph{A pointwise estimate for fractionary derivatives with applications to partial differential equations},  Proc.\ Natl.\ Acad.\ Sci.\ USA~\textbf{100} (2003), no.~26, 15316--15317.\bibitem[CD]{CD}J.R.~Cannon and E.~DiBenedetto, \emph{The initial value problem for the  {B}oussinesq equations with data in {$L^{p}$}}, Approximation methods for  {N}avier-{S}tokes problems ({P}roc.\ {S}ympos., {U}niv.\ {P}aderborn,  {P}aderborn, 1979), Lecture Notes in Math., vol.~771, Springer, Berlin, 1980,  pp.~129--144. \bibitem[CF]{CF}P.~Constantin and C.~Foias, \emph{Navier-{S}tokes equations}, Chicago  Lectures in Mathematics, University of Chicago Press, Chicago, IL, 1988.\bibitem[CG]{CG}M.~Chen and O.~Goubet, \emph{Long-time asymptotic behavior of  two-dimensional dissipative {B}oussinesq systems}, Discrete Contin.\ Dyn.\ Syst.\ Ser.~S~\textbf{2} (2009), no.~1, 37--53. \bibitem[CLR]{CLR}P.~Constantin, M.~Lewicka, and L.~Ryzhik, \emph{Travelling waves in  two-dimensional reactive {B}oussinesq systems with no-slip boundary  conditions}, Nonlinearity~\textbf{19} (2006), no.~11, 2605--2615. \bibitem[CN]{CN}D.~Chae and H.-S.~Nam, \emph{Local existence and blow-up criterion for  the {B}oussinesq equations}, Proc.\ Roy.\ Soc.\ Edinburgh Sect.~A~\textbf{127}  (1997), no.~5, 935--946. \bibitem[CW]{CW}C.~Cao and J.~Wu, \emph{Global regularity for the two-dimensional  anisotropic {B}oussinesq equations with vertical dissipation}, Arch.\ Ration.\ Mech.\ Anal.~\textbf{208} (2013), no.~3, 985--1004. \bibitem[DG]{DG}C.R.~Doering and J.D.~Gibbon, \emph{Applied analysis of the  {N}avier-{S}tokes equations}, Cambridge Texts in Applied Mathematics,  Cambridge University Press, Cambridge, 1995. \bibitem[DP1]{DP1}R.~Danchin and M.~Paicu, \emph{Les th\'eor\`emes de {L}eray et de  {F}ujita-{K}ato pour le syst\`eme de {B}oussinesq partiellement visqueux},  Bull.\ Soc.\ Math.\ France~\textbf{136} (2008), no.~2, 261--309. \bibitem[DP2]{DP2}R.~Danchin and M.~Paicu, \emph{Les th\'eor\`emes de {L}eray et de  {F}ujita-{K}ato pour le syst\`eme de {B}oussinesq partiellement visqueux},  Bull.\ Soc.\ Math.\ France~\textbf{136} (2008), no.~2, 261--309. \bibitem[ES]{ES}W.~E and C.-W.~Shu, \emph{Small-scale structures in {B}oussinesq  convection}, Phys.\ Fluids~\textbf{6} (1994), no.~1, 49--58. \bibitem[FMT]{FMT}C.~Foias, O.~Manley, and R.~Temam, \emph{Modelling of the interaction of small  and large eddies in two-dimensional turbulent flows}, RAIRO Mod\'el.\ Math.  Anal.\ Num\'er.~\textbf{22} (1988), no.~1, 93--118. \bibitem[HK1]{HK1}T.~Hmidi and S.~Keraani,   \emph{On the global well-posedness of the two-dimensional {B}oussinesq system with a zero diffusivity},   Adv.\ Differential Equations~\textbf{12} (2007), no.~4, 461--480. \bibitem[HK2]{HK2}T.~Hmidi and S.~Keraani,   \emph{On the global well-posedness of the {B}oussinesq system with zero viscosity},   Indiana Univ.\ Math.~J.~\textbf{58} (2009), no.~4, 1591--1618. \bibitem[HKR]{HKR}T.~Hmidi, S.~Keraani, and F.~Rousset, \emph{Global well-posedness for  {E}uler-{B}oussinesq system with critical dissipation}, Comm.\ Partial  Differential Equations~\textbf{36} (2011), no.~3, 420--445. \bibitem[HL]{HL}T.Y. Hou and C. Li, \emph{Global well-posedness of the viscous  {B}oussinesq equations}, Discrete Contin. Dyn. Syst.~\textbf{12} (2005),  no.~1, 1--12. \bibitem[HKZ1]{HKZ1}W.~Hu, I.~Kukavica, and M.~Ziane,   \emph{On the regularity for the {B}oussinesq equations in a bounded domain},   J.~Math.\ Phys.~\textbf{54} (2013), no.~8, 081507, 10. \bibitem[HKZ2]{HKZ2}   W.~Hu, I.~Kukavica, and M.~Ziane,   \emph{Persistence of regularity for the viscous {B}oussinesq equations with zero diffusivity},   Asymptot.\ Anal.~\textbf{91} (2015), no.~2, 111--124. \bibitem[HS]{HS}   F.~Hadadifard and A.~Stefanov,   \emph{On the global regularity of the 2D critical Boussinesq system with $\alpha >2/3$},   Comm. Math. Sci.~\textbf{15} (2017), no.~5, 1325--1351.\bibitem[JMWZ]{JMWZ} Q.~Jiu, C.~Miao, J.~Wu, and Z.~Zhang,   \emph{The two-dimensional incompressible {B}oussinesq equations with general critical dissipation},   SIAM J. Math. Anal. \textbf{46} (2014), no.~5, 3426--3454.\bibitem[J]{J} N.~Ju,   \emph{The maximum principle and the global attractor for the dissipative 2{D} quasi-geostrophic equations}, \  Comm.\ Math.\ Phys.~\textbf{255} (2005), no.~1, 161--181. \bibitem[KP]{KP} T.~Kato and G.~Ponce, \emph{Commutator estimates and the {E}uler and  {N}avier-{S}tokes equations}, Comm.\ Pure Appl.\ Math.~\textbf{41} (1988),  no.~7, 891--907. \bibitem[KTW]{KTW}J.P.~Kelliher, R.~Temam, and X.~Wang, \emph{Boundary layer  associated with the {D}arcy-{B}rinkman-{B}oussinesq model for convection in  porous media}, Phys.\ D~\textbf{240} (2011), no.~7, 619--628. \bibitem[KWZ]{KWZ} I.~Kukavica, F.~Wang and M.~Ziane, \emph{Persistence of regularity for solutions of the Boussinesq equations in Sobolev spaces}, Adv.\ Differential Equations~\textbf{21} (2016), no.~1/2, 85--108.\bibitem[LLT]{LLT}A.~Larios, E.~Lunasin, and E.S.~Titi, \emph{Global well-posedness  for the 2{D} {B}oussinesq system with anisotropic viscosity and without heat  diffusion}, J.~Differential Equations~\textbf{255} (2013), no.~9, 2636--2654.\bibitem[LPZ]{LPZ}M.-J.~Lai, R.~Pan, and K.~Zhao, \emph{Initial boundary value problem  for two-dimensional viscous {B}oussinesq equations}, Arch.\ Ration.\ Mech.  Anal.~\textbf{199} (2011), no.~3, 739--760. \bibitem[R]{R} J.C.~Robinson, \emph{Infinite-dimensional dynamical systems}, Cambridge  Texts in Applied Mathematics, Cambridge University Press, Cambridge, 2001, An  introduction to dissipative parabolic PDEs and the theory of global  attractors. \bibitem[SW]{SW}   A.~Stefanov and J.~Wu,   \emph{A global regularity result for the 2D Boussinesq equation with critical dissipation}, to appear,  Journal d Analyse Math\'ematique.\bibitem[T1]{T1}R.~Temam, \emph{Infinite-dimensional dynamical systems in mechanics and  physics}, second ed., Applied Mathematical Sciences, vol.~68,  Springer-Verlag, New York, 1997. \bibitem[T2]{T2}R.~Temam, \emph{Navier-{S}tokes equations}, AMS Chelsea Publishing,  Providence, RI, 2001, Theory and numerical analysis, Reprint of the 1984  edition. \bibitem[T3]{T3}R.~Temam, \emph{Navier-{S}tokes equations and nonlinear functional  analysis}, second ed., CBMS-NSF Regional Conference Series in Applied  Mathematics, vol.~66, Society for Industrial and Applied Mathematics (SIAM),  Philadelphia, PA, 1995. \end{thebibliography}
\end{document}